\documentclass{article}

\usepackage{microtype}
\usepackage{graphicx}

\usepackage{booktabs} 

\usepackage{amssymb}
\usepackage{mathrsfs}
\usepackage{amsfonts}
\usepackage{amsmath}
\usepackage{graphicx}
\usepackage{multirow}
\usepackage{float}
\usepackage{booktabs}
\usepackage{algorithmic}
\usepackage{color}
\usepackage{tabu}

\usepackage{graphics}
\usepackage{amssymb}%
\usepackage{url}

\usepackage{epstopdf}
\usepackage{subcaption}
\usepackage[pdftex,pdfstartview=FitB]{hyperref}
\usepackage{esint}
\usepackage{bm}
\usepackage{bbm}
\usepackage{url}            
\usepackage{booktabs}       
\usepackage{nicefrac}       
\usepackage{microtype}      
\usepackage{color}
\usepackage{tabu}
\usepackage{comment}
\usepackage[standard,hyperref,amsmath]{ntheorem}
\usepackage{graphicx}
\usepackage{enumitem}
\usepackage{textcomp}
\usepackage{setspace}
\usepackage{cleveref}
\usepackage{latexsym}
\usepackage[round]{natbib}
\usepackage{alltt}
\usepackage{xspace}

\usepackage[accepted]{icml2019}

\newtheorem{assumption}{Assumption}

\newcommand{\calG}{\mathcal{G}}
\newcommand{\calC}{\mathcal{C}}
\newcommand{\calD}{\mathcal{D}}
\newcommand{\calB}{\mathcal{B}}

\newcommand{\calL}{\mathcal{L}}

\newcommand{\calR}{\mathcal{R}}
\newcommand{\calS}{\mathcal{S}}
\newcommand{\calO}{\mathcal{O}}
\newcommand{\calP}{\mathcal{P}}
\newcommand{\calQ}{\mathcal{Q}}
\newcommand{\calW}{\mathcal{W}}
\newcommand{\calU}{\mathcal{U}}
\newcommand{\calV}{\mathcal{V}}

\newcommand{\One}{\bm{1}}
\newcommand{\zero}{\bm{0}}
\newcommand{\RR}{\mathbb{R}}

\newcommand{\Rp}{\RR_+}

\newcommand{\NN}{\mathbb{N}}

\newcommand{\e}{\mathrm{e}}

\newcommand{\bX}{\bm{X}}
\newcommand{\bx}{\bm{x}}
\newcommand{\bY}{\bm{Y}}
\newcommand{\by}{\bm{y}}

\DeclareMathOperator*{\argmin}{\mathrm{argmin}}

\newcommand{\dom}{\mathrm{dom}}

\newcommand{\bv}{\bm{v}}
\newcommand{\bw}{\bm{w}}

\newcommand{\bU}{\bm{U}}
\newcommand{\bV}{\bm{V}}
\newcommand{\bW}{\bm{W}}

\newcommand{\bs}{\bm{s}}
\newcommand{\bu}{\bm{u}}

\newcommand{\bg}{\bm{g}}

\newcommand{\bcL}{\overline{\calL}}
\newcommand{\KL}{K\L\xspace}
\newcommand{\res}{\mathrm{res}}

\numberwithin{equation}{section}

\icmltitlerunning{Global Convergence of Block Coordinate Descent in Deep Learning}

\begin{document}

\twocolumn[
\icmltitle{Global Convergence of Block Coordinate Descent in Deep Learning}



\icmlsetsymbol{equal}{*}

\begin{icmlauthorlist}
\icmlauthor{Jinshan Zeng}{jxnu,hkust,equal}
\icmlauthor{Tim Tsz-Kit Lau}{nu,equal}
\icmlauthor{Shao-Bo Lin}{cityuhk}
\icmlauthor{Yuan Yao}{hkust}
\end{icmlauthorlist}

\icmlaffiliation{jxnu}{School of Computer and Information Engineering, Jiangxi Normal University, Nanchang 330022, Jiangxi, China}
\icmlaffiliation{nu}{Department of Statistics, Northwestern University, Evanston, IL 60208, USA}
\icmlaffiliation{cityuhk}{Department of Mathematics, City University of Hong Kong, Kowloon, Hong Kong. Part of this work was done while Tim Tsz-Kit Lau was at Department of Mathematics, The Hong Kong University of Science and Technology}
\icmlaffiliation{hkust}{Department of Mathematics, The Hong Kong University of Science and Technology, Clear Water Bay, Kowloon, Hong Kong}

\icmlcorrespondingauthor{Yuan Yao}{yuany@ust.hk}


\vskip 0.3in
]



\printAffiliationsAndNotice{\icmlEqualContribution} 

\begin{abstract}
Deep learning has aroused extensive attention due to its great empirical success. The efficiency of the block coordinate descent (BCD) methods has been recently demonstrated in deep neural network (DNN) training. However, theoretical studies on their convergence properties are limited due to the highly nonconvex nature of DNN training. In this paper, we aim at providing a general methodology for provable convergence guarantees for this type of methods. In particular, for most of the commonly used DNN training models involving both two- and three-splitting schemes, we establish the global convergence to a critical point at a rate of ${\cal O}(1/k)$, where $k$ is the number of iterations. The results extend to general loss functions which have Lipschitz continuous gradients and deep residual networks (ResNets). Our key development adds several new elements to the Kurdyka-{\L}ojasiewicz inequality framework that enables us to carry out the global convergence analysis of BCD in the general scenario of deep learning.
\end{abstract}


\section{Introduction}
\label{sc:introduction}
Tremendous research activities have been dedicated to deep learning due to its great success in some real-world applications
such as image classification in computer vision \citep{Hinton-imagenet-2012}, speech
recognition \citep{Hinton-speech-2012,Sainath-speech-2013},
statistical machine translation \citep{Devlin-NLP-2014}, and especially outperforming human in Go games \citep{AlphaGo_nature}.

The practical optimization algorithms for training neural networks can be mainly divided into three categories in terms
of the amount of first- and second-order information used, namely,
gradient-based, (approximate) second-order and
gradient-free methods.
Gradient-based methods make use of
backpropagation \citep{Hinton-BP1986} to compute gradients of network parameters. Stochastic gradient descent (SGD) method proposed by
\citet{Robbins-SGD1951} serve as the basis.
Much of research endeavour is devoted to adaptive variants of vanilla SGD in recent years, including AdaGrad \citep{Duchi-adagrad-2011}, RMSProp \citep{Hinton-RMSProp-2012}, Adam \citep{Kingma2015} and AMSGrad \citep{Reddi2018}.
(Approximate) second-order methods mainly include Newton's method \citep{LeCun-normalization-1998}, L-BFGS and conjugate gradient \citep{Le-BFGS-2011}. Despite the great success of these gradient-based methods, they may suffer from the vanishing gradient issue for training deep networks \citep{Goodfellow-et-al-2016}.
As an alternative to overcome this issue, gradient-free methods have been recently adapted to the DNN training, including (but not limited to)
block coordinate descent (BCD) methods \citep{Carreira2014-MAC,Zhang-BCD-NIPS2017,Lau2018,Askari-BCD2018,Gu-BCD2018} and alternating direction method of multipliers (ADMM) \citep{Goldstein-ADMM-DNN2016,Zhang2016-ADMM-DNN}.
The main reasons for the surge of attention of these two algorithms are twofold.
One reason is that they are gradient-free,
and thus are able to deal with non-differentiable nonlinearities and potentially avoid the vanishing gradient issue \citep{Goldstein-ADMM-DNN2016,Zhang-BCD-NIPS2017}.
As shown in \Cref{Fig:acc_deep}, it is observed that vanilla SGD fails to train a ten-hidden-layer MLPs while BCD still works and achieves a moderate accuracy within a few epochs.
The other reason is that BCD and ADMM can be easily implemented in a distributed and parallel manner \citep{Boyd-DADMM2011,Mahajan2017}, therefore in favour of distributed/decentralized scenarios.

\begin{figure}[h!]
	\centering
	\begin{subfigure}[h]{0.49\columnwidth}
		\includegraphics[width=1\columnwidth]{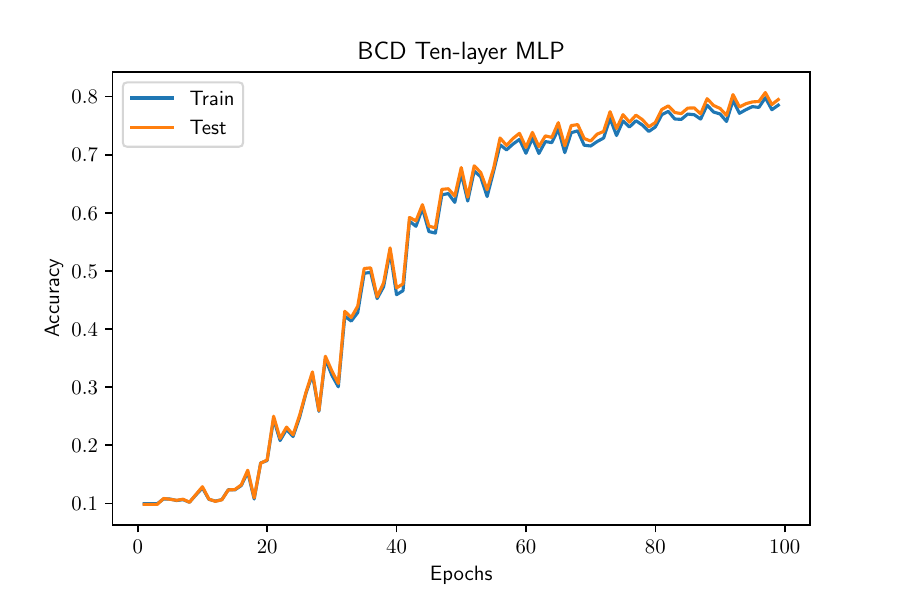}
		\caption{BCD}
		\label{Fig:acc_train_deep_mnist}
	\end{subfigure}
	\begin{subfigure}[h]{0.49\columnwidth}
		\includegraphics[width=1\columnwidth]{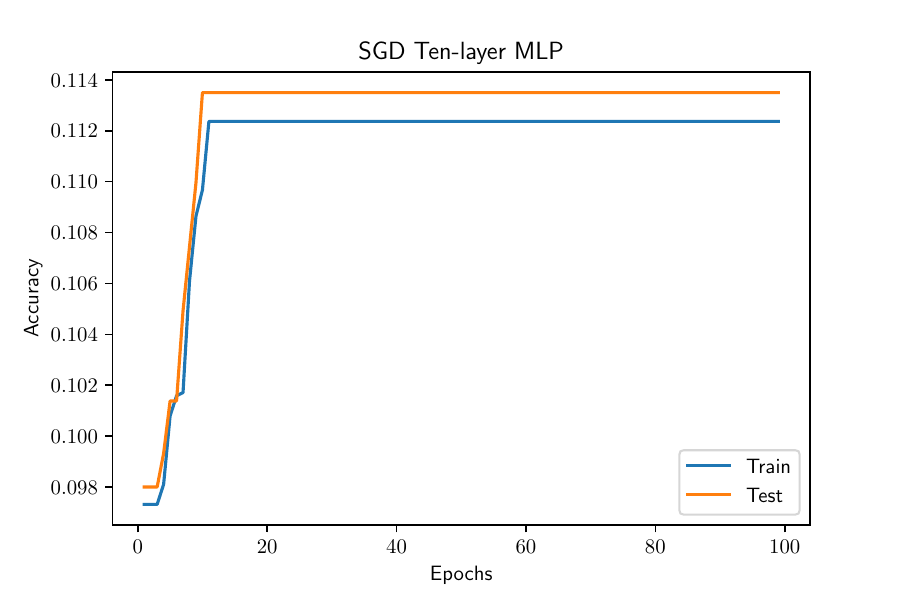}
		\caption{SGD}
		\label{Fig:acc_test_deep_mnist}
	\end{subfigure}
	\caption{Comparison of training and test accuracies of BCD and SGD for training ten-hidden-layer MLPs on the MNIST dataset.
Refer to \Cref{sc:BCD-SGD-10layerMLP} for details of this experiment\protect\footnotemark.}
	\label{Fig:acc_deep}
	\vspace*{-8mm}
\end{figure}

The BCD methods currently adopted in DNN training run into two categories depending on the specific formulations of the objective functions, namely,
the \textbf{two-splitting formulation} and \textbf{three-splitting formulation} (shown in \ref{Eq:dnn-admm} and \ref{Eq:dnn-admm1}), respectively.
Examples of the two-splitting formulation include
\citet{Carreira2014-MAC,Zhang-BCD-NIPS2017,Askari-BCD2018,Gu-BCD2018},
whilst \citet{Goldstein-ADMM-DNN2016,Lau2018} adopt the three-splitting formulation.
Convergence studies of BCD methods appeared recently in more restricted settings.
In \citet{Zhang-BCD-NIPS2017}, a BCD method was suggested to solve the Tikhonov regularized deep neural network training problem using a lifting trick to avoid the computational hurdle imposed by ReLU.
\footnotetext{Codes available at: \url{https://github.com/timlautk/BCD-for-DNNs-PyTorch}.}
Its convergence was established through the framework of \citet{Xu-Yin-BCD2013},
where the \textit{block multiconvexity}\footnote{A function $f$ with multi-block variables $(\bx_1,\ldots, \bx_p)$ is called \textit{block multiconvex} if it is convex with respect to each block variable when fixing the other blocks, and $f$ is called \textit{blockwise Lipschitz differentiable} if it is differentiable with respect to each block variable and its gradient is Lipschitz continuous while fixing the others.}
and \textit{differentiability} of the unregularized part of the objective function
play central roles in the analysis.
However, for other commonly used activations such as sigmoid, the convergence analysis of \citet{Xu-Yin-BCD2013} cannot be directly applied since the \textit{block multiconvexity} may be violated.
\citet{Askari-BCD2018} and \citet{Gu-BCD2018} extended the lifting trick introduced by \citet{Zhang-BCD-NIPS2017} to deal with a class of strictly increasing and invertible activations, and then adapted BCD methods to solve the lifted DNN training models.
However, no convergence guarantee was provided in both \citet{Askari-BCD2018} and \citet{Gu-BCD2018}.
Following the similar lifting trick as in \citet{Zhang-BCD-NIPS2017},
\citet{Lau2018} proposed a proximal BCD based on the three-splitting formulation of the regularized DNN training problem with ReLU activation.
The global convergence was also established through the analysis framework of \citet{Xu-Yin-BCD2013}.
However, similar convergence results for other commonly used activation functions are still lacking.

In this paper, we aim to fill these gaps.
Our main contribution is to provide a general methodology to establish the global convergence\footnote{\textit{Global convergence} refers to the case that the algorithm converges starting from any finite initialization.}
of these BCD methods in the common DNN training settings, without requiring the \textit{block multiconvexity} and \textit{differentiability} assumptions as in \citet{Xu-Yin-BCD2013}.
Instead, our key assumption is the Lipschitz continuity of the activation on any bounded set (see \Cref{Assumption-model}(b)).
Specifically, \Cref{Thm:BCD-ConvThm1} establishes the global convergence to a critical point at an $\calO(1/k)$ rate of the BCD methods using the proximal strategy, while extensions to the prox-linear strategy for general losses are provided in \Cref{Coro:globalconv-prox-linear1} and to residual networks (ResNets) are shown in \Cref{Coro:BCD-resnet}. Our assumptions are applicable to most cases appeared in the literature. Specifically in \Cref{Thm:BCD-ConvThm1}, if the loss function, activations, and convex regularizers are lower semicontinuous and either real-analytic (see \Cref{Def:real-analytic}) or semialgebraic (see \Cref{Def:semialgebraic}), and the activations are Lipschitz continuous on any bounded set,
then BCD converges to a critical point at an $\calO(1/k)$ rate starting from any finite initialization, where $k$ is the number of iterations.
Note that these assumptions are satisfied by most commonly used DNN training models,
where (a) the loss function can be any of the squared, logistic, hinge, exponential or cross-entropy losses,
(b) the activation function can be any of ReLU, leaky ReLU, sigmoid, tanh, linear, polynomial, or softplus functions,
and (c) the regularizer can be any of the squared $\ell_2$ norm, squared Frobenius norm, the elementwise $1$-norm, or the sum of squared Frobenuis norm and elementwise $1$-norm (say, in the vector case, the elastic net by \citealp{Zou-elastic-net-2005}), or the indicator function of the nonnegative closed half space or a closed interval (see \Cref{Proposition:special-cases}).

Our analysis is based on the Kurdyka-{\L}ojasiewicz (\KL) inequality \citep{Lojasiewicz-KL1993,Kurdyka-KL1998}
framework formulated in \citet{Attouch2013}. However there are several different treatments compared to the state-of-the-art work \citep{Xu-Yin-BCD2013} that enables us to achieve the general convergence guarantee aforementioned. According to \citet[Theorem 2.9]{Attouch2013},
the \textit{sufficient descent}, \textit{relative error} and \textit{continuity} conditions, together with the \textit{\KL} assumption yield the global convergence of a nonconvex algorithm.
In order to obtain the \textit{sufficient descent} condition, we exploit the proximal strategy for all non-strongly convex subproblems (see \Cref{alg:BCD-3-split} and \Cref{Lemm:BCD-suff-desc}), without requiring the block multiconvexity assumption used in \citet[Lemma 2.6]{Xu-Yin-BCD2013}.
In order to establish the \textit{relative error} condition, we use the Lipschitz continuity of the activation functions and perform some careful treatments on the specific updates of the BCD methods (see \Cref{Lemm:BCD-grad-bound}),
without requiring the (locally) Lipschitz differentiability of the unregularized part as used in \citet[Lemma 2.6]{Xu-Yin-BCD2013}.
The \textit{continuity} condition is established via the lower semicontinuity assumptions of the loss, activations and regularizers.
The treatments of this paper are of their own value to the optimization community.
The detailed comparisons between this paper and the existing literature can be found in \Cref{sc:proof-ideas}.

The rest of this paper is organized as follows.
\Cref{sc:BCD} describes the BCD methods when adapted to the splitting formulations of DNN training problems.
\Cref{sc:BCD-convergence} establishes their global convergence results, followed by some extensions.
\Cref{sc:proof-ideas} illustrates the key ideas of proof with some discussions.
We conclude this paper in \Cref{sc:conclusion}.

\section{DNN training via BCD}
\label{sc:BCD}
In this section, we describe the specific forms of BCD involving both two- and three-splitting formulations.

\subsection{DNN training with variable splitting}
\label{sc:DNN-variable-split}

Consider $N$-layer feedforward neural networks with $N-1$ hidden layers of the neural networks.
Particularly, let $d_i \in \mathbb{N}$ be the number of hidden units in the $i$-th hidden layer for $i=1,\ldots,N-1$.
Let $d_0$ and $d_N$ be the number of units of input and output layers, respectively.
Let $\bW_i \in \RR^{d_i \times d_{i-1}}$ be the weight matrix between the $(i-1)$-th layer and the $i$-th layer for any $i=1,\ldots N$.\footnote{To  simplify notations, we regard the input and output layers as the $0$-th and $N$-th layers, respectively, and absorb the bias of each layer into $\bW_i$.} Let ${\cal Z}:= \{(\bx_j, \by_j)\}_{j=1}^n {\subset \RR^{d_0} \times \RR^{d_N}}$ be $n$ samples, where $\by_j$'s are the one-hot vectors of labels.
Denote $\calW:=\{\bW_i\}_{i=1}^N$, $\bX:= (\bx_1, \bx_2, \ldots, \bx_n) \in \RR^{d_0 \times n}$ and $\bY:= (\by_1, \by_2, \ldots, \by_n) \in \RR^{d_N \times n}$.
With the help of these notations, the DNN training problem can be formulated as the following empirical risk minimization:
\begin{equation}
\label{Eq:dnn-org}
\min_{\calW} \calR_n(\Phi(\bX;\calW), \bY), 
\end{equation}
where $\calR_n\left( \Phi(\bX;\calW), \bY\right)~:=~\frac{1}{n} \sum_{j=1}^n \ell\left( \Phi(\bx_j;\calW),\by_j\right)$, $\ell: \RR^{d_N} \times \RR^{d_N} \rightarrow \Rp \cup \{0\}$ is some loss function,
$\Phi(\bx_j;\calW) = \sigma_N(\bW_N\sigma_{N-1}(\bW_{N-1}\cdots \bW_2\sigma_1(\bW_1\bx_j))$ is the neural network model with $N$ layers and weights $\calW$ and $\sigma_i$ is the activation function of the $i$-th layer (generally, $\sigma_N \equiv \mathrm{Id}$, i.e., the identity function) and $\calR_n$ is called the empirical risk (also known as the training loss).

Note that the DNN training model \eqref{Eq:dnn-org} is highly nonconvex as the variables are coupled via the deep neural network architecture, which brings many challenges for the design of efficient training algorithms and also its theoretical analysis.
To make Problem \eqref{Eq:dnn-org} more computationally tractable, variable splitting
is one of the most commonly used ways \citep{Goldstein-ADMM-DNN2016,Zhang-BCD-NIPS2017,Askari-BCD2018,Gu-BCD2018,Lau2018}.
The main idea of variable splitting is to transform a complicated problem (where the variables are coupled highly nonlinearly) into a relatively simpler one (where the variables are coupled much looser) via introducing some additional variables.

\subsubsection{Two-splitting formulation.}

Considering general deep neural network architectures, the DNN training problem can be naturally formulated as the following model (called \textbf{two-splitting formulation})\footnote{Here we consider the regularized DNN training model. The model reduces to the original DNN training model \eqref{Eq:dnn-org} without regularization. }:
\begin{align}
\vspace*{-2mm}
\min_{\calW, \calV} \calL_0\left(\calW, \calV\right) &:= \calR_n(\bV_N;\bY) + \sum_{i=1}^N r_i(\bW_i) + \sum_{i=1}^N s_i(\bV_i) \nonumber\\
\text{subject to} \quad \bV_i &= \sigma_i(\bW_i\bV_{i-1}), \quad i=1,\ldots, N,
\label{Eq:dnn-admm}
\end{align}
where $\calR_n(\bV_N; \bY):= \frac{1}{n} \sum_{j=1}^n \ell\left( (\bV_N)_{:j}, \by_j\right) $ denotes the empirical risk, $\calV:=\{\bV_i\}_{i=1}^N$,
$(\bV_N)_{:j}$ is the $j$-th column of $\bV_N$.
In addition, $r_i$ and $s_i$ are extended-real-valued, nonnegative functions revealing the priors of the weight variable $\bW_i$ and the state variable $\bV_i$ (or the constraints on $\bW_i$ and $\bV_i$) for each $i=1,\ldots N$, and define
$\bV_0:= \bX$.
In order to solve the {two-splitting formulation} \eqref{Eq:dnn-admm}, the following alternative minimization problem was suggested in the literature: 
\begin{equation}
\label{Eq:2-split-min}
\min_{\calW, \calV} \calL\left(\calW, \calV\right) := \calL_0 \left(\calW, \calV\right) + \frac{\gamma}{2}  \sum_{i=1}^N \|\bV_i - \sigma_i(\bW_i \bV_{i-1})\|_F^2,
\end{equation}
where $\gamma>0$ is a hyperparameter\footnote{In \eqref{Eq:3-split-min}, we use a uniform hyperparameter $\gamma$ for the sum of all quadratic terms for the simplicity of notation. In practice, $\gamma$ can be different for each quadratic term and our proof still goes through.}.

The DNN training model \eqref{Eq:dnn-admm} can be very general, where:
(a) $\ell$ can be the squared, logistic, hinge, cross-entropy or other commonly used loss functions;
(b) $\sigma_i$ can be ReLU, leaky ReLU, sigmoid, linear, polynomial, softplus or other commonly used activation functions;
(c) $r_i$ can be the squared $\ell_2$ norm, the $\ell_1$ norm, the elastic net \citep{Zou-elastic-net-2005}, the indicator function of some nonempty closed convex set\footnote{The indicator function $\iota_\calC$ of a nonempty convex set $\calC$ is defined as $\iota_\calC(x) = 0$ if $x\in\calC$ and $+\infty$ otherwise. } (such as the nonnegative closed half space or a closed interval $[0,1]$);
(d) $s_i$ can be the $\ell_1$ norm \citep{Ji-SparDBN2014}, the indicator function of some convex set with simple projection \citep{Zhang-BCD-NIPS2017}.
Particularly, if there is no regularizer or constraint on $\bW_i$ (or $\bV_i$), then $r_i$ (or $s_i$) can be zero.

The network architectures considered in this paper exhibit generality to various types of DNNs, including but not limited to the fully (or sparse) connected MLPs \citep{Rosenblatt-MLP-1961}, convolutional neural networks \citep[CNNs;][]{Fukushima1980,LeCun-CNN1998} and residual neural networks \citep[ResNets;][]{He-ResNet-2016}. For CNNs, the weight matrix $\bW_i$ is sparse and shares some symmetry structures represented as permutation invariants, which are linear constraints and up to a linear reparameterization, so all the main results below are still valid.

Various existing BCD algorithms for DNN training \citep{Carreira2014-MAC,Zhang-BCD-NIPS2017,Askari-BCD2018,Gu-BCD2018} can be regarded as special cases in terms of the use of the two-splitting formulation \eqref{Eq:dnn-admm}.
In fact, \citet{Carreira2014-MAC} considered a specific DNN training model with squared loss and sigmoid activation function, and proposed the method of auxiliary coordinate (MAC) based on the two-splitting formulation of DNN training \eqref{Eq:dnn-admm}, as a two-block BCD method with the weight variables $\calW$ as one block and the state variables $\calV$ as the other. For each block, a nonlinear least squares problem is solved by some iterative methods. Furthermore, \citet{Zhang-BCD-NIPS2017} proposed a BCD type method for DNN training with ReLU and squared loss. To avoid the computational hurdle imposed by ReLU, the DNN training model was relaxed to a smooth multiconvex formulation via lifting ReLU into a higher dimensional space \citep{Zhang-BCD-NIPS2017}.
Such a relaxed BCD is in fact a special case of {two-splitting formulation} \eqref{Eq:2-split-min} with $\sigma_i \equiv \mathrm{Id}, r_i \equiv 0, s_i (\bV_i) = \iota_{\cal X}(\bV_i)$, $i=1,\ldots,N$, where ${\cal X}$ is the nonnegative closed half-space with the same dimension of $\bV_i$,
while \citet{Askari-BCD2018} and \citet{Gu-BCD2018} extended such lifting trick
to more general DNN training settings,
of which the activation function can be not only ReLU, but also sigmoid and leaky ReLU.
The general formulations studied in these two papers are also special cases of the {two-splitting formulation} with different $\sigma_i, r_i$ and $s_i$ for $i=1,\ldots,N$.

\subsubsection{Three-splitting formulation.}
Note that the variables $\bW_i$ and $\bV_{i-1}$ are coupled by the nonlinear activation function in the $i$-th constraint of the \text{two-splitting formulation} \eqref{Eq:dnn-admm},
which may bring some difficulties and challenges for solving problem \eqref{Eq:dnn-admm} efficiently,
particularly, when the activation function is ReLU.
Instead,
the following \textbf{three-splitting formulation} was used in \citet{Goldstein-ADMM-DNN2016,Lau2018}:
\begin{multline}
\label{Eq:dnn-admm1}
\min_{\calW, \calV, \calU}  \calL_0\left( \calW, \calV \right) \quad \text{subject to}\\ \; \bU_i = \bW_i \bV_{i-1}, \; \bV_i = \sigma_i(\bU_i), \quad i=1,\ldots,N,
\end{multline}
where $\calU:= \{\bU_i\}_{i=1}^N$. From \eqref{Eq:dnn-admm1}, the variables are coupled much more loosely, particularly for variables $\bW_i$ and $\bV_{i-1}$.
As described later, such a \text{three-splitting formulation} can be beneficial to designing some more efficient methods,
though $N$ extra auxiliary variables $\bU_i$'s are introduced. Similarly, the following alternative unconstrained problem
was suggested in the literature:
\begin{multline}
\hspace{-4mm}\min_{\calW, \calV, \calU} \bcL\left(\calW, \calV, \calU\right) := \calL_0\left(\calW, \calV \right) \\ +\frac{\gamma}{2}\sum_{i=1}^N \left[ \|\bV_i - \sigma_i(\bU_i)\|_F^2 + \|\bU_i - \bW_i \bV_{i-1}\|_F^2 \right]. \label{Eq:3-split-min}
\end{multline}

\subsection{Description of BCD algorithms}
\label{sc:BCD-description}

In the following, we describe how to adapt the BCD method to Problems \eqref{Eq:2-split-min} and \eqref{Eq:3-split-min}.
The main idea of the BCD method of Gauss-Seidel type for a minimization problem with multi-block variables
is to update all the variables cyclically while fixing the remaining blocks at their last updated values \citep{Xu-Yin-BCD2013}.
In this paper, we consider the BCD method with the backward order (but not limited to this as discussed later) for the updates of variables,
i.e., the variables are updated from the output layer to the input layer,
and for each layer, we update the variables $\{\bV_i, \bW_i\}$ cyclically for Problem \eqref{Eq:2-split-min} as well as the variables $\{\bV_i, \bU_i, \bW_i\}$ cyclically for Problem \eqref{Eq:3-split-min}.
Since $\sigma_N \equiv \mathrm{Id}$, the output layer is paid special attention.
Particularly, for most blocks, we adopt the proximal update strategies for two major reasons: (1) To practically stabilize the training process; (2) To yield the desired ``sufficient descent'' property for theoretical justification.
For each subproblem, we assume that its minimizer can be achieved.
The BCD algorithms for Problems \eqref{Eq:2-split-min} and \eqref{Eq:3-split-min} can be summarized in \Cref{alg:BCD-2-split,alg:BCD-3-split}, respectively.

\begin{algorithm}[h]
	{\footnotesize
		\begin{algorithmic}\caption{Two-splitting BCD for DNN Training
\eqref{Eq:2-split-min}}\label{alg:BCD-2-split}
			\STATE {\bf Data}: $\bX \in \RR^{d_0 \times n}$, $\bY \in \RR^{d_N \times n}$
			\STATE {\bf Initialization}: $\{\bW_i^0, \bV_i^0\}_{i=1}^N$, $\bV_0^k \equiv \bV_0 := \bX$
			\STATE {\bf Parameters:} $\gamma>0$, $\alpha>0$ \footnotemark
			\smallskip
			\FOR{$k=1,\ldots$}
			\STATE $\bV_N^k = \argmin_{\bV_N} \ \{ s_N(\bV_N) + {\cal R}_n(\bV_N; \bY) +  \frac{\gamma}{2} \|\bV_N - \bW_N^{k-1}\bV_{N-1}^{k-1}\|_F^2+ \frac{\alpha}{2} \|\bV_N - \bV_N^{k-1}\|_F^2\} $
			\STATE $\bW_N^k =\argmin_{\bW_N} \ \{ r_N(\bW_N)+\frac{\gamma}{2} \|\bV_N^k - \bW_N\bV_{N-1}^{k-1}\|_F^2+ \frac{\alpha}{2}\|\bW_N - \bW_N^{k-1}\|_F^2\}$
			\FOR{$i= N-1,\ldots, 1$}
			\STATE $\bV_i^k =\argmin_{\bV_i} \ \{s_i(\bV_i) + \frac{\gamma}{2} \|\bV_i - \sigma_i(\bW_i^{k-1}\bV_{i-1}^{k-1})\|_F^2 + \frac{\gamma}{2}\|\bV_{i+1}^k - \sigma_{i+1}(\bW_{i+1}^k \bV_i)\|_F^2 + \frac{\alpha}{2} \|\bV_i - \bV_i^{k-1}\|_F^2\}$
			\STATE $\bW_i^k =\argmin_{\bW_i} \{r_i(\bW_i) + \frac{\gamma}{2} \|\bV_i^k - \sigma_i(\bW_i\bV_{i-1}^{k-1})\|_F^2 + \frac{\alpha}{2} \|\bW_i - \bW_i^{k-1}\|_F^2 \}$
			\ENDFOR
			\ENDFOR
	\end{algorithmic}}
\end{algorithm}
\footnotetext{In practice, $\gamma$ and $\alpha$ can vary among blocks and our proof still goes through.}

\begin{algorithm}[h]
	{\small
		\begin{algorithmic}\caption{Three-splitting BCD for DNN training
 \eqref{Eq:3-split-min}}\label{alg:BCD-3-split}
			\STATE {\bf Samples}: $\bX \in \RR^{d_0 \times n}$, $\bY \in \RR^{d_N \times n}$
			\STATE {\bf Initialization}: $\{\bW_i^0, \bV_i^0, \bU_i^0\}_{i=1}^N$, $\bV_0^k \equiv \bV_0 := \bX$
			\STATE {\bf Parameters:} $\gamma>0$, $\alpha>0$
			\smallskip
			\FOR{$k=1,\ldots$}
			\STATE $\bV_N^k = \argmin_{\bV_N} \ \{s_N(\bV_N) + {\cal R}_n(\bV_N; \bY) +  \frac{\gamma}{2} \|\bV_N - \bU_N^{k-1}\|_F^2 + \frac{\alpha}{2} \|\bV_N - \bV_N^{k-1}\|_F^2\}$
			\STATE $\bU_N^k = \argmin_{\bU_N} \ \{\frac{\gamma}{2}\|\bV_N^k - \bU_N\|_F^2 + \frac{\gamma}{2}\|\bU_N - \bW_N^{k-1}\bV_{N-1}^{k-1}\|_F^2\}$
			\STATE $\bW_N^k = \argmin_{\bW_N} \ \{ r_N(\bW_N)+\frac{\gamma}{2} \|\bU_N^k - \bW_N\bV_{N-1}^{k-1}\|_F^2+ \frac{\alpha}{2}\|\bW_N - \bW_N^{k-1}\|_F^2\}$
			\FOR{$i= N-1,\ldots, 1$}
			\STATE $\bV_i^k = \argmin_{\bV_i} \ \{s_i(\bV_i) + \frac{\gamma}{2} \|\bV_i - \sigma_i(\bU_i^{k-1})\|_F^2 + \frac{\gamma}{2}\|\bU_{i+1}^k - \bW_{i+1}^k \bV_i\|_F^2 \}$
			\STATE $\bU_i^k = \argmin_{\bU_i}\ \{\frac{\gamma}{2}\|\bV_i^k - \sigma_i(\bU_i)\|_F^2 + \frac{\gamma}{2}\|\bU_i - \bW_i^{k-1}\bV_{i-1}^{k-1}\|_F^2 + \frac{\alpha}{2}\|\bU_i -\bU_i^{k-1}\|_F^2\}$
			\STATE $\bW_i^k = \argmin_{\bW_i} \{r_i(\bW_i) + \frac{\gamma}{2} \|\bU_i^k - \bW_i\bV_{i-1}^{k-1}\|_F^2 + \frac{\alpha}{2} \|\bW_i - \bW_i^{k-1}\|_F^2 \}$
			\ENDFOR
			\ENDFOR
	\end{algorithmic}}
\end{algorithm}

One major merit of \Cref{alg:BCD-3-split} over \Cref{alg:BCD-2-split} is that in each subproblem, almost all updates are simple proximal updates\footnote{For $\bV_N^k$-update, we can regard $s_N(\bV_N) + {\cal R}_n(\bV_N; \bY)$ as a new proximal function $\tilde{s}_N(\bV_N)$. } (or just least squares problems), which usually have closed form solutions to many commonly used DNNs,
while a drawback of \Cref{alg:BCD-3-split} over \Cref{alg:BCD-2-split} is that more storage memory is required due to the introduction of additional variables $\{\bU_i\}_{i=1}^N$.
Some typical examples leading to the closed form solutions include:
(a) $r_i, s_i$ are $0$ (i.e., no regularization), or the squared $\ell_2$ norm (a.k.a. \textit{weight decay}), or the indicator function of a nonempty closed convex set with a simple projection like the nonnegative closed half space and the closed interval $[0,1]$;
(b) the loss function $\ell$ is the squared loss
or hinge loss (see \Cref{Lemm:solution-hinge-min} in \Cref{app:prox});
and (c) $\sigma_i$ is ReLU (see \Cref{Lemm:solution-relu-min} in \Cref{app:proof_lemma_1}), leaky ReLU, or linear link function.
For other cases in which $r_i$ and $s_i$ are the $\ell_1$ norm, $\sigma_i$ is the sigmoid function, and the loss $\ell$ is the logistic function,
the associated subproblems can be also solved cheaply via some efficient existing methods.
Discussions on specific implementations of these BCD methods can be referred to \Cref{sc:BCD-implementation}.

\section{Global convergence analysis of BCD}
\label{sc:BCD-convergence}

In this section, we establish the global convergence of both \Cref{alg:BCD-2-split} for Problem \eqref{Eq:2-split-min}, and  \Cref{alg:BCD-3-split} for Problem \eqref{Eq:3-split-min}, followed by some extensions.

\subsection{Main assumptions}
\label{sc:Assumptions}

First of all, we present our main assumptions, which involve the definitions of \textit{real analytic} and \textit{semialgebraic} functions.

Let $h:\RR^p \to \RR \cup \{+\infty\}$ be an extended-real-valued function (respectively, $h:\RR^p \rightrightarrows \RR^q$ be a point-to-set mapping), its \textit{graph} is defined by
\begin{align*}
&\mathrm{Graph}(h) := \{(\bx,y)\in \RR^p \times \RR: y = h(\bx)\}, \\
(\text{resp.}\; &\mathrm{Graph}(h) := \{(\bx,\by)\in \RR^p \times \RR^q: \by \in h(\bx)\}),
\end{align*}
and its domain by $\dom(h):=\{\bx\in \RR^p: h(\bx)<+\infty\}$ (resp. $\dom(h) :=\{\bx\in\RR^p: h(\bx)\neq \varnothing\}$).
When $h$ is a proper function, i.e., when $\dom(h) \neq \varnothing,$ the set of its global minimizers (possibly empty) is denoted by
\[
\argmin h:=\{\bx\in \RR^p: h(\bx) = \inf h\}.
\]

\begin{definition}[Real analytic]
	\label{Def:real-analytic}
	A function $h$ with domain an open set $U\subset \RR$ and range the set of either all real or complex numbers, is said to be \textbf{real analytic} at $u$ if the function $h$ may be represented by a convergent power series on some interval of positive radius centered at $u$, i.e.,
	$h(x) = \sum_{j=0}^{\infty} \alpha_j(x-u)^j,$
	for some $\{\alpha_j\} \subset \RR$.
	The function is said to be \textbf{real analytic} on $V\subset U$ if it is real analytic at each $u \in V$
    \citep[Definition 1.1.5]{Krantz2002-real-analytic}.
	The real analytic function $f$ over $\RR^p$ for some positive integer $p>1$ can be defined similarly.
	
According to \citet{Krantz2002-real-analytic}, typical real analytic functions include polynomials, exponential functions, and the logarithm, trigonometric and power functions  on any open set of their domains.
One can verify whether a multivariable real function $h(\bx)$ on $\RR^p$ is analytic by checking the analyticity of $g(t):= h(\bx+ t \by)$ for any $\bx, \by \in \RR^p$.
\end{definition}

\begin{definition}[Semialgebraic]\hfill
	\label{Def:semialgebraic}
	\begin{enumerate}
		\item[(a)]
		A set $\calD \subset \RR^p$ is called semialgebraic \citep{Bochnak-semialgebraic1998} if it can be represented as
		\[
		\calD = \bigcup_{i=1}^s \bigcap_{j=1}^t \left\lbrace \bx\in \RR^p: P_{ij}(\bx) = 0, Q_{ij}(\bx)>0 \right\rbrace ,
		\]
		where $P_{ij}, Q_{ij}$ are real polynomial functions for $1\leq i \leq s, 1\leq j \leq t.$
		
		\item[(b)]
		A function $h:\RR^p\rightarrow \RR\cup \{+\infty\}$ (resp. a point-to-set mapping $h:\RR^p \rightrightarrows \RR^q$) is called \textit{semialgebraic} if its graph $\mathrm{Graph}(h)$ is semialgebraic.
	\end{enumerate}

\end{definition}

According to \citet{Lojasiewicz1965-semianalytic,Bochnak-semialgebraic1998} and \citet[I.2.9, page 52]{Shiota1997-subanalytic}, the class of semialgebraic sets are stable under the operation of finite union, finite intersection, Cartesian product or complementation. Some typical examples include \text{polynomial} functions, the indicator function of a semialgebraic set, and the \text{Euclidean norm}  \citep[page 26]{Bochnak-semialgebraic1998}.

\begin{assumption}
\label{Assumption-model}
Suppose that
\begin{enumerate}
\item[(a)]
the loss function $\ell$ is a proper lower semicontinuous\footnote{A function $f: {\cal X} \rightarrow \mathbb{R}$ is called \textit{lower semicontinuous} if $\mathop{\lim\inf}_{x \rightarrow x_0}\ f(x) \geq f(x_0)$ for any $x_0 \in {\cal X}$. } and nonnegative function,

\item[(b)]
the activation functions $\sigma_i$ ($i=1\dots, N-1$) are Lipschitz continuous on any bounded set,

\item[(c)]
the regularizers $r_i$ and $s_i$ ($i=1,\ldots,N$) are nonegative lower semicontinuous convex functions, and

\item[(d)]
all these functions $\ell$, $\sigma_i$, $r_i$ and $s_i$ ($i=1,\ldots,N$) are either real analytic or semialgebraic, and continuous on their domains.
\end{enumerate}
\end{assumption}

According to \citet{Krantz2002-real-analytic,Lojasiewicz1965-semianalytic,Bochnak-semialgebraic1998} and \citet[I.2.9, page 52]{Shiota1997-subanalytic},
most of the commonly used DNN training models \eqref{Eq:dnn-admm} can be verified to satisfy \Cref{Assumption-model} as shown in the following proposition,
the proof of which is provided in \Cref{app:proof-proposition1}.

\begin{proposition}
\label{Proposition:special-cases}
Examples satisfying \Cref{Assumption-model} include:
\begin{enumerate}
\item[(a)] $\ell$ is the squared, logistic, hinge, or cross-entropy losses;
\item[(b)] $\sigma_i$ is ReLU, leaky ReLU, sigmoid, hyperbolic tangent, linear, polynomial, or softplus activations;
\item[(c)] $r_i$ and $s_i$ are the squared $\ell_2$ norm, the $\ell_1$ norm, the elastic net, the indicator function of some nonempty closed convex set (such as the nonnegative closed half space, box set or a closed interval $[0,1]$),
or $0$ if no regularization.
\end{enumerate}
\end{proposition}

\subsection{Main theorem}
\label{sc:main-convergence-results}

Under Assumption \ref{Assumption-model}, we state our main theorem as follows.

\begin{theorem}
\label{Thm:BCD-ConvThm1}
Let $\{\calQ^k:=\left(\{\bW_i^k\}_{i=1}^N, \{\bV_i^k\}_{i=1}^N\right)\}_{k\in\NN}$ and $\{\calP^k:=\left(\{\bW_i^k\}_{i=1}^N, \{\bV_i^k\}_{i=1}^N, \{\bU_i^k\}_{i=1}^N\right)\}_{k\in\NN}$ be the sequences generated by \Cref{alg:BCD-2-split,alg:BCD-3-split}, respectively.
Suppose that Assumption \ref{Assumption-model} holds,
and that one of the following conditions holds: (i) there exists a convergent subsequence $\{\calQ^{k_j}\}_{j\in \NN}$ (resp. $\{\calP^{k_j}\}_{j\in\NN}$); (ii) $r_i$ is coercive\footnote{An extended-real-valued function $h: \RR^p \rightarrow \RR\cup \{+\infty\}$ is called coercive if and only if $h(\bx) \rightarrow +\infty$ as $\|\bx\| \rightarrow +\infty$. } for any $i=1,\ldots, N$; (iii) $\cal L$ (resp. $\bcL$) is coercive.
Then for any $\alpha>0$, $\gamma>0$ and any finite initialization $\calQ^0$ (resp. $\calP^0$), the following hold
\begin{enumerate}[label=(\alph*)]

\item
$\{\calL(\calQ^k)\}_{k\in\NN}$ (resp. $\{\bcL(\calP^k)\}_{k\in\NN}$) converges to some ${\calL}^*$ (resp. $\bcL^*$).

\item
 $\{\calQ^k\}_{k\in\NN}$ (resp. $\{\calP^k\}_{k\in\NN}$) converges to a critical point of $\calL$ (resp. $\bcL$).

\item
$\frac{1}{K}\sum_{k=1}^K \|\bg^k\|_F^2 \to 0$ at the rate $\calO(1/K)$ where $\bg^k \in \partial \calL(\calQ^k)$.
Similarly, $\frac{1}{K}\sum_{k=1}^K \|{\bar{\bg}}^k\|_F^2 \to 0$ at the rate $\calO(1/K)$ where ${\bar{\bg}}^k \in \partial {\bcL}(\calP^k)$.
\end{enumerate}
\end{theorem}
Note that the DNN training problems \eqref{Eq:2-split-min} and \eqref{Eq:3-split-min} in this paper generally do not satisfy such a \textit{Lispchitz differentiable} property, particularly, when ReLU activation is used. Compared to the existing literature, this theorem establishes the global convergence without the \textit{block multiconvexity} and \textit{Lipschitz differentiability} assumptions used in \citet{Xu-Yin-BCD2013}, which are often violated by the DNN training problems due to the nonlinearity of the activations.



\subsection{Extensions}
\label{sc:extensions}

We extend the established convergence results to the BCD methods for general losses with the prox-linear strategy, and the BCD methods for training ResNets.

\subsubsection{Extension to prox-linear}
\label{sc:BCD-prox-linear}

Note that in the $\bV_N$-update of both  \Cref{alg:BCD-2-split,alg:BCD-3-split}, the empirical risk is involved in the optimization problems. It is generally hard to obtain its closed-form solution except for some special cases such as the case that the loss is the square loss. For other smooth losses such as the logistic, cross-entropy, and exponential losses, we suggest using the following \textit{prox-linear} update strategies, that is, for some parameter $\alpha>0$, the $\bV_N$-update in \Cref{alg:BCD-2-split} is
	\begin{multline}
	\label{Eq:VN-alg1}
	\hspace{-5mm}\bV_N^k =
	\argmin_{\bV_N} \left\lbrace s_N(\bV_N) + \langle \nabla \calR_n(\bV_N^{k-1}; \bY), \bV_N - \bV_N^{k-1}\rangle \right. \\\left. + \frac{\alpha}{2} \|\bV_N-\bV_N^{k-1}\|_F^2  + \frac{\gamma}{2}\|\bV_N-\bW_N^{k-1}\bV_{N-1}^{k-1}\|_F^2\right\rbrace,
	\end{multline}
and the $\bV_N$-update in \Cref{alg:BCD-3-split} is
	\begin{multline}
	\label{Eq:VN-alg2}
	\hspace{-5mm}\bV_N^k =
	\argmin_{\bV_N} \left\lbrace s_N(\bV_N) + \langle \nabla \calR_n(\bV_N^{k-1}; \bY), \bV_N - \bV_N^{k-1}\rangle\right. \\\left.  + \frac{\alpha}{2} \|\bV_N - \bV_N^{k-1}\|_F^2 + \frac{\gamma}{2}\|\bV_N - \bU_N^{k-1}\|_F^2\right\rbrace.
	\end{multline}
From \eqref{Eq:VN-alg1} and \eqref{Eq:VN-alg2}, when $s_N$ is zero or its proximal operator can be easily computed, then $\bV_N$-updates for both BCD algorithms can be implemented with explicit expressions.
Therefore, the specific uses of these BCD methods are very flexible, mainly depending on users'  understanding of their own problems.

The claims in \Cref{Thm:BCD-ConvThm1} still hold for the prox-linear updates adopted for the $\bV_N$-updates if the loss is smooth with Lipschitz continuous gradient, as stated in the following \Cref{Coro:globalconv-prox-linear1}.
\begin{theorem}
[Global convergence for prox-linear update]
\label{Coro:globalconv-prox-linear1}
Consider adopting the prox-linear updates \eqref{Eq:VN-alg1}, \eqref{Eq:VN-alg2} for the $\bV_N$-subproblems in \Cref{alg:BCD-2-split,alg:BCD-3-split}, respectively. Under the conditions of \Cref{Thm:BCD-ConvThm1}, if further $\nabla {\cal R}_n$ is Lipschitz continuous with a Lipschitz constant $L_R$ and $\alpha > \max\left\lbrace 0,\frac{L_R-\gamma}{2}\right\rbrace $, then all claims in Theorem \ref{Thm:BCD-ConvThm1} still hold for both algorithms.
\end{theorem}
The proof of \Cref{Coro:globalconv-prox-linear1} is presented in \Cref{sc:proof-convergence-BCD-ResNets}.
It establishes the global convergence of a BCD method for the commonly used DNN training models with nonlinear losses, such as logistic or cross-entropy losses, etc.
Equipped with the prox-linear strategy, all updates of BCD can be implemented easily and allow large-scale distributed computations.

\subsubsection{Extension to ResNet Training}
\label{sc:resnet}

In this section, we first adapt the BCD method to the residual networks \citep[ResNets;][]{He-ResNet-2016}, and then extend the established convergence results of BCD to this case.
Without loss of generality, similar to \eqref{Eq:dnn-admm}, we consider the following simplified ResNets training problem,
\begin{multline}
\label{Eq:dnn-resnet}
\min_{\calW, \calV} \  \calR_n(\bV_N; \bY) + \sum_{i=1}^N r_i(\bW_i) + \sum_{i=1}^N s_i(\bV_i) \\
\text{s.t.} \quad \bV_i - \bV_{i-1} = \sigma_i(\bW_i \bV_{i-1}), \quad i=1,\ldots, N.
\end{multline}
Since the ReLU activation is usually used in ResNets,
we only consider the \textbf{three-splitting formulation} of \eqref{Eq:dnn-resnet}:
\begin{multline*}
\min_{\calW, \calV, \calU} \ \calR_n(\bV_N;\bY) + \sum_{i=1}^N r_i(\bW_i) + \sum_{i=1}^N s_i(\bV_i)\\
\text{s.t.} \ \bU_i = \bW_i \bV_{i-1}, \ \bV_i - \bV_{i-1}= \sigma_i(\bU_i), \ i=1,\ldots,N, \nonumber
\end{multline*}
and then adapt BCD to the following minimization problem,
\begin{align}
\label{Eq:QP-BCD-resnet}
\min_{\calW, \calV, \calU} \quad \bcL_{\res}({\cal W},{\cal V},{\cal U}),
\end{align}
where ${\cal W}:=\{\bW_i\}_{i=1}^N$, ${\cal V}:=\{\bV_i\}_{i=1}^N$, ${\cal U}:=\{\bU_i\}_{i=1}^N$ as defined before, and
\begin{align*}
&\bcL_{\res}({\cal W},{\cal V},{\cal U})
:=\calR_n(\bV_N; \bY)+ \sum_{i=1}^N r_i(\bW_i) + \sum_{i=1}^N s_i(\bV_i)\\
&+ \frac{\gamma}{2}\sum_{i=1}^N \left[ \|\bV_i - \bV_{i-1} - \sigma_i(\bU_i)\|_F^2 + \|\bU_i - \bW_i \bV_{i-1}\|_F^2 \right]. \nonumber
\end{align*}
When applied to \eqref{Eq:QP-BCD-resnet}, we use the same update order of \Cref{alg:BCD-3-split} but slightly change the subproblems according to the objective $\bcL_{\res}$ in \eqref{Eq:QP-BCD-resnet}.
The specific BCD algorithm for ResNets is presented in \Cref{alg:BCD-resnet} in \Cref{sc:proof-convergence-BCD-ResNets}.

Similarly, we establish the convergence of BCD for the DNN training model with ResNets \eqref{Eq:QP-BCD-resnet} as follows.
\begin{theorem}[Convergence of BCD for ResNets] \label{Coro:BCD-resnet} \hfill\\
Let $\{\{\bW_i^k, \bV_i^k, \bU_i^k\}_{i=1}^N\}_{k\in\NN}$ be a sequence generated by BCD for the DNN training model with ResNets (i.e., \Cref{alg:BCD-resnet}).
Let assumptions of \Cref{Thm:BCD-ConvThm1} hold.
Then all claims in Theorem \ref{Thm:BCD-ConvThm1} still hold for BCD with ResNets by replacing $\bcL$ with $\bcL_{\res}$.

Moreover, consider adopting the prox-linear update for the $\bV_N$-subproblem in \Cref{alg:BCD-resnet}, then under the assumptions of \Cref{Coro:globalconv-prox-linear1}, all claims of \Cref{Coro:globalconv-prox-linear1} still hold for \Cref{alg:BCD-resnet}.
\end{theorem}

The proof of this theorem is presented in \Cref{sc:proof-convergence-BCD-ResNets}.
ResNets is one of the most popular network architectures used in the deep learning community and has profound applications in computer vision.
How to efficiently train ResNets is thus very important, especially since it is not of a fully-connected structure.
This theorem, for the first time, shows that the BCD method might be a good candidate for the training of ResNets with global convergence guarantee.

\section{Keystones and discussions}
\label{sc:proof-ideas}
In this section, we present the keystones of our proofs followed by some discussions.

\subsection{Main ideas of proofs}
Our proofs follow the analysis framework formulated in \citet{Attouch2013}, where the establishments of the \textit{sufficient descent} and \textit{relative error} conditions and the verifications of the \textit{continuity} condition and \textit{\KL} property of the objective function are the four key ingredients.
In order to establish the \textit{sufficient descent} and \textit{relative error} properties, two kinds of assumptions, namely, \textbf{(a)} multiconvexity and differentiability assumption, and \textbf{(b)} (blockwise) Lipschitz differentiability assumption on the unregularized part of objective function are commonly used in the literature,
where \citet{Xu-Yin-BCD2013} mainly used \textbf{assumption (a)},
and the literature \citep{Attouch2013,Xu-Yin-ncvxBCD-2017,Bolte2014} mainly used \textbf{assumption (b)}.
Note that in our cases, the unregularized part of $\calL$ in \eqref{Eq:2-split-min},
\begin{align*}
\calR_n(\bV_N; \bY) + \frac{\gamma}{2}  \sum_{i=1}^N \|\bV_i - \sigma_i(\bW_i \bV_{i-1})\|_F^2,
\end{align*}
and that of
$\bcL$ in \eqref{Eq:3-split-min},
\[\calR_n(\bV_N; \bY)+ \frac{\gamma}{2}\sum_{i=1}^N \left[ \|\bV_i - \sigma_i(\bU_i)\|_F^2 + \|\bU_i - \bW_i \bV_{i-1}\|_F^2 \right]\]
usually do not satisfy any of the block multiconvexity and differentiability assumptions (i.e., \textbf{assumption (a)}), and the blockwise Lipschitz differentiability assumption (i.e., \textbf{assumption (b)}).
For instance, when $\sigma_i$ is ReLU or leaky ReLU,
the functions $\|\bV_i - \sigma_i(\bW_i \bV_{i-1})\|_F^2$ and $\|\bV_i - \sigma_i(\bU_i)\|_F^2$ are non-differentiable and nonconvex with respect to $\bW_i$-block and $\bU_i$-block, respectively.

In order to overcome these challenges, in this paper,
we first exploit the proximal strategies for all the non-strongly convex subproblems (see \Cref{alg:BCD-3-split}) to cheaply obtain the desired sufficient descent property (see \Cref{Lemm:BCD-suff-desc}),
and then take advantage of the Lipschitz continuity of the activations as well as the specific splitting formulations to yield the desired relative error property (see \Cref{Lemm:BCD-grad-bound}). Below we present these two key lemmas, while leaving other details in Appendix (where the verification of the \KL property for the concerned DNN training models satisfying \Cref{Assumption-model} can be referred to \Cref{Thm:KL-property} in \Cref{sc:KL}, and the verification of the \textit{continuity} condition is shown by \eqref{Eq:continuity-cond} in \Cref{app:proof-thm5}).
Based on \Cref{Lemm:BCD-suff-desc,Lemm:BCD-grad-bound}, \Cref{Thm:KL-property} and \eqref{Eq:continuity-cond}, we prove \Cref{Thm:BCD-ConvThm1} according to \citet[Theorem 2.9]{Attouch2013}, with details shown in \Cref{sc:proof-main-theorem}.

\subsection{Sufficient descent lemma}

We state the established sufficient descent lemma as follows.

\begin{lemma}[Sufficient descent]
	\label{Lemm:BCD-suff-desc}
    Let $\{{\cal P}^k\}_{k\in\NN}$ be a sequence generated by the BCD method (\Cref{alg:BCD-3-split}). Then, under the assumptions of \Cref{Thm:BCD-ConvThm1},
	\begin{align}
	\bcL({\cal P}^k)  \leq \overline{\calL}(\calP^{k-1}) - a\|\calP^k - \calP^{k-1}\|_F^2,
	\end{align}
    for some constant $a>0$ specified in the proof. 
\end{lemma}
From \Cref{Lemm:BCD-suff-desc}, the Lagrangian sequence $\{\bcL({\cal P}^k)\}$ is monotonically decreasing, and the descent quantity of each iterate can be lower bounded by the discrepancy between the current iterate and its previous iterate.
This lemma is crucial for the global convergence of a nonconvex algorithm.
It tells at least the following four important items:
\textbf{(i)} $\{\bcL({\cal P}^k)\}_{k\in\NN}$ is convergent if $\bcL$ is lower bounded;
\textbf{(ii)} $\{{\cal P}^k\}_{k\in\NN}$ itself is bounded if $\bcL$ is coercive and ${\cal P}^0$ is finite;
\textbf{(iii)} $\{\calP^k\}_{k\in\NN}$ is square summable, i.e., $\sum_{k=1}^{\infty} \|\calP^k - \calP^{k-1}\|_F^2 < \infty$,
implying its asymptotic regularity, i.e., $\|\calP^k - \calP^{k-1}\|_F \to 0$ as $k\to \infty$;
and
\textbf{(iv)} $\frac{1}{K} \sum_{k=1}^{K} \|\calP^k - \calP^{k-1}\|_F^2 \to 0$ at a rate of $\calO(1/K)$.
Leveraging \Cref{Lemm:BCD-suff-desc}, we can establish the global convergence (i.e., the whole sequence convergence) of BCD in DNN training settings. In contrast, \citet{Davis2018} only establish the subsequence convergence of SGD in DNN training settings.
Such a gap between the subsequence convergence of SGD in \citet{Davis2018} and the whole sequence convergence of BCD in this paper exists mainly because SGD can only achieve the descent property but not the sufficient descent property.

It can be noted from \Cref{Lemm:BCD-suff-desc} that neither \textit{multiconvexity and differentiability} nor \textit{Lipschitz differentiability} assumptions are imposed on the DNN training models to yield this lemma, as required in the literature \citep{Xu-Yin-BCD2013,Attouch2013,Xu-Yin-ncvxBCD-2017,Bolte2014}.
Instead, we mainly exploit the proximal strategy for all non-strongly convex subproblems in \Cref{alg:BCD-3-split} to establish this lemma.

\subsection{Relative error lemma}

We now present the obtained relative error lemma.

\begin{lemma}[Relative error]
	\label{Lemm:BCD-grad-bound}
	Under the conditions of \Cref{Thm:BCD-ConvThm1},
	let $\cal B$ be an upper bound of $\calP^{k-1}$ and $\calP^k$ for any positive integer $k$,
	$L_\calB$ be a uniform Lipschitz constant of $\sigma_i$ on the bounded set $\{\calP: \|\calP\|_F \leq \calB\}$. Then for any positive integer $k$, it holds that,
	\begin{align*}
    \|\bar{\bg}^k\|_F \leq \bar{b} \|\calP^k - \calP^{k-1}\|_F, \quad \bar{\bg}^k \in \partial \overline{\calL}({\cal P}^k)
	\end{align*}
	for some constant $\bar{b}>0$ specified later in the proof,
    where
	\begin{align*}
	\partial \overline{\calL}({\cal P}^k)
	:= (\{\partial_{\bW_i}\overline{\calL}\}_{i=1}^N, \{\partial_{\bV_i} \overline{\calL}\}_{i=1}^N, \{\partial_{\bU_i} \overline{\calL} \}_{i=1}^N)(\calP^k).
	\end{align*}	
\end{lemma}
\Cref{Lemm:BCD-grad-bound} shows that the subgradient sequence of the Lagrangian is upper bounded by the discrepancy between the current and previous iterates.
Together with the asymptotic regularity of $\{\calP^k\}_{k\in\NN}$ yielded by \Cref{Lemm:BCD-suff-desc}, \Cref{Lemm:BCD-grad-bound} shows the critical point convergence.
Also, together with the claim \textbf{(iv)} implied by \Cref{Lemm:BCD-suff-desc}, namely, the ${\cal O}(1/K)$ rate of convergence of $\frac{1}{K} \sum_{k=1}^{K} \|\calP^k - \calP^{k-1}\|_F^2 \to 0$,
\Cref{Lemm:BCD-grad-bound} yields the ${\cal O}(1/K)$ rate of convergence (to a critical point) of BCD, i.e., $\frac{1}{K} \sum_{k=1}^K \|\bar{\bg}^k\|_F \rightarrow 0$ at the rate of $\calO(1/K)$.

From \Cref{Lemm:BCD-grad-bound}, both differentiability and (blockwise) Lipschitz differentiability assumptions are not imposed.
Instead, we only use the Lipschitz continuity (on any bounded set) of the activations, which is a very mild and natural condition satisfied by most commonly used activation functions.
In order to achieve this lemma, we also need to do some special treatments on the specific updates of BCD algorithms as demonstrated in \Cref{app:estab-relative-error}.

\section{Conclusion}
\label{sc:conclusion}
The empirical efficiency of BCD methods in deep neural network (DNN) training has been demonstrated in the literature.
However, the theoretical understanding of their convergence is still very limited and it lacks a general framework due to the fact that DNN training is a highly nonconvex problem.
In this paper, we fill this void by providing a general methodology to establish the global convergence of the BCD methods for a class of DNN training models,
which encompasses most of the commonly used BCD methods in the literature as special cases.
Under some mild assumptions, we establish the global convergence at a rate of $\calO(1/k)$, with $k$ being the number of iterations, to a critical point of the DNN training models with several variable splittings. Our theory is also extended to residual networks with general losses which have Lipschitz continuous gradients. Such work may lay a theoretical foundation of BCD methods for their applications to deep learning.

\section*{Acknowledgments}
The work of Jinshan Zeng is supported in part by the National Natural Science Foundation (NNSF) of China (No.~61603162, 61876074).
The work of Tim Tsz-Kit Lau is supported by the University Fellowship for first-year Ph.D. students by The Graduate School at Northwestern University.
The work of Shao-Bo Lin is supported in part by the NNSF of China (No.~61876133).
The work of Yuan Yao is supported in part by HKRGC grant 16303817, 973 Program of China (No.~2015CB85600), NNSF of China (No.~61370004, 11421110001), as well as grants from Tencent AI Lab, Si Family Foundation, Baidu BDI, and Microsoft Research-Asia. The authors thank Prof.~Jian Peng from UIUC for helpful discussions.

\bibliography{reference}
\bibliographystyle{icml2019}

\newpage
\onecolumn
\appendix
\section*{Appendix}
\section{Implementations on BCD methods}
\label{sc:BCD-implementation}
In this section, we provide several remarks to discuss the specific implementations of BCD methods. Reproducible PyTorch codes can be found at: \url{https://github.com/timlautk/BCD-for-DNNs-PyTorch} or \url{https://github.com/yao-lab/BCD-for-DNNs-PyTorch}.

\begin{remark}[On the initialization of parameters]
\label{remark:initialization}
In practice, the weights $\{\bW_i\}_{i=1}^N$  are generally initialized according to some Gaussian distributions with small standard deviations.
The bias vectors are usually set as all one vectors scaled by some small constants.
Given the weights and bias vectors,
the auxiliary variables $\{\bU_i\}_{i=1}^N$ and state variables $\{\bV_i\}_{i=1}^N$ are usually initialized by a single forward pass through the network.
\end{remark}

\begin{remark}[On the update order]
\label{remark:update-order}
We suggest such a backward update order in this paper due to the nested structure of DNNs.
Besides the update order presented in \Cref{alg:BCD-3-split},
any arbitrary deterministic update order can be incorporated into the BCD methods, and our proofs still go through.
\end{remark}

\begin{remark}[On the distributed implementation]
\label{remark:distributed}
One major advantage of BCD is that it can be implemented in distributed and parallel manner like in ADMM.
Specifically, given $m$ servers, the total training data are distributed to these servers. Denotes $S_j$ as the subset of samples at server $j$. Thus, $n = \sum_{j=1}^m \sharp (S_j)$, where $\sharp (S_j)$ denotes the cardinality of $S_j$.
For each layer $i$, the state variable $\bV_i$ is divided into $m$ submatrices by column,
i.e., $\bV_i := ((\bV_i)_{:S_1}, \ldots, (\bV_i)_{:S_m})$, where $(\bV_i)_{:S_j}$ denotes the submatrix of $V_i$ including all the columns in the index set $S_j$.
The auxiliary variables $\bU_i$'s are decomposed similarly.
From \Cref{alg:BCD-3-split}, the updates of $\{\bV_i\}_{i=1}^N$ and $\{\bU_i\}_{i=1}^N$ do not need any communication and thus, can be computed in a parallel way.
The difficult part is the update of weight $\bW_i$, which is generally hard to parallelize.
To deal with this part, there are some effective strategies suggested in the literature like \citet{Goldstein-ADMM-DNN2016}.
\end{remark}

\section{Proof of \Cref{Proposition:special-cases}}
\label{app:proof-proposition1}
\begin{proof}
We verify these special cases as follows.
	
\textbf{On the loss function $\ell$:}
Since these losses are all nonnegative and continuous on their domains, they are proper lower semicontinuous and lower bounded by $0$.
In the following, we only verify that they are either real analytic or semialgebraic.
	\begin{enumerate}[label=(a\arabic*)]
		\item If $\ell(t)$ is the squared ($t^2$) or exponential ($\e^t$) loss, then according to \citet{Krantz2002-real-analytic}, they are real analytic.
		
		\item \label{a2} If $\ell(t)$ is the logistic loss ($\log(1+\e^{-t})$), since it is a composition of logarithm and exponential functions which both are real analytic, thus according to \Cref{Lemm:real-analytic}, the logistic loss is real analytic.
		
		\item If $\ell(\bu;\by)$ is the cross-entropy loss, i.e., given $\by\in \RR^{d_N}$, $\ell(\bu; \by) = -\frac{1}{d_N}[\langle \by, \log\widehat{\by}(\bu)\rangle + \langle \One-\by, \log(\One-\widehat{\by}(\bu))\rangle]$, where $\log$ is performed elementwise and $\left( \widehat{\by}(\bu)_i\right) _{1\le i \le d_N} := \left( (1+\e^{-u_i})^{-1}\right)_{1\le i \le d_N} $ for any $\bu \in \RR^{d_N}$, which can be viewed as a linear combination of logistic functions, then by \ref{a2} and \Cref{Lemm:real-analytic}, it is also analytic.
		
		\item
		If $\ell$ is the hinge loss, i.e., given $\by\in \RR^{d_N}$, $\ell(\bu;\by) := \max\{0,1-\left\langle \bu, \by\right\rangle \}$ for any $\bu \in \RR^{d_N}$, by  \Cref{Lemm:semialgbraic}(1), it is semialgebraic, because its graph is $\mathrm{cl}(\calD)$, the closure of the set $\calD$, where
		\begin{align*}
		{\cal D} = \{(\bu,z): 1-\left\langle \bu, \by\right\rangle -z=0, \bm{1}-\bu \succ 0\}\cup \{(\bu,z): z=0, \left\langle \bu, \by\right\rangle -1>0\}.
		\end{align*}
	\end{enumerate}
	
	\textbf{On the activation function $\sigma_i$:}
    Since all the considered specific activations are continuous on their domains, they are Lipschitz continuous on any bounded set.
    In the following, we only need to check that they are either real analytic or semialgebraic.
	\begin{enumerate}[label=(b\arabic*)]
		\item
		If $\sigma_i$ is a linear or polynomial function, then according to \citet{Krantz2002-real-analytic}, $\sigma_i$ is real analytic.
		
		\item
		If $\sigma_i(t)$ is sigmoid, $(1+\e^{-t})^{-1}$, or hyperbolic tangent, $\tanh t:=\frac{\e^t - \e^{-t}}{\e^t + \e^{-t}}$, then the sigmoid function is a composition $g\circ h$ of these two functions where $g(u) = \frac{1}{1+u}, u>0$ and $h(t) = \e^{-t}$ (resp.~$g(u) = 1-\frac{2}{u+1}, u>0$ and $h(t)=\e^{2t}$ in the hyperbolic tangent case). According to \citet{Krantz2002-real-analytic}, $g$ and $h$ in both cases are real analytic. Thus, according to \Cref{Lemm:real-analytic}, sigmoid and hyperbolic tangent functions are real analytic.
		
		\item
		If $\sigma_i$ is ReLU, i.e., $\sigma_i(u):=\max\{0,u\}$, then we can show that ReLU is semialgebraic
		since its graph is $\mathrm{cl}({\cal D})$, the closure of the set ${\cal D}$, where
		\begin{align*}
		{\cal D} =
		\{(u,z): u-z=0, u>0\}\cup \{(u,z): z=0, -u>0\}.
		\end{align*}
		
		\item
		Similar to the ReLU case, if $\sigma_i$ is leaky ReLU, i.e., $\sigma_i(u) = u$ if $u>0$, otherwise $\sigma_i(u) = a u$ for some $a>0$, then we can similarly show that leaky ReLU is semialgebraic since its graph is $\mathrm{cl}(\calD)$, the closure of the set $\calD$, where
		\begin{align*}
		{\cal D} =
		\{(u,z): u-z=0, u>0\}\cup \{(u,z): au-z=0, -u>0\}.
		\end{align*}
		
		\item
		If $\sigma_i$ is polynomial as used in \citet{Poggio-landscape2017}, then according to \citet{Krantz2002-real-analytic}, it is real analytic.

        \item
		If $\sigma_i$ is softplus, i.e., $\sigma_i(u) = \frac{1}{t}\log(1+\e^{tu})$ for some $t>0$, since it is a composition of two analytic functions $\frac{1}{t}\log(1+u)$ and $\e^{tu}$, then according to \citet{Krantz2002-real-analytic}, it is real analytic.
	\end{enumerate}
	
	\textbf{On $r_i(\bW_i)$, $s_i(\bV_i)$:}
By the specific forms of these regularizers, they are nonnegative, lower semicontinuous and continuous on their domain.
In the following, we only need to verify they are either real analytic and semialgebraic.
	\begin{enumerate}[label=(c\arabic*)]
		\item \textbf{the squared $\ell_2$ norm $\|\cdot\|_2^2$:}
		According to \citet{Bochnak-semialgebraic1998}, the $\ell_2$ norm is semialgebraic, so is its square according to \Cref{Lemm:semialgbraic}(2), where $g(t)=t^2$ and $h(\bW)=\|\bW\|_2$.
		
		\item \textbf{the squared Frobenius norm $\|\cdot\|_F^2$:} The squared Frobenius norm is semiaglebraic since it is a finite sum of several univariate squared functions.
		
		\item \textbf{the elementwise $1$-norm $\|\cdot\|_{1,1}$:}
		Note that $\|\bW\|_{1,1} = \sum_{i,j} |\bW_{ij}|$ is the finite sum of absolute functions $h(t)=|t|$. According to \Cref{Lemm:semialgbraic}(1), the absolute value function is semialgebraic since its graph is the closure of the following semialgebraic set
		\[
		{\cal D}=\{(t,s): t+s =0, -t>0\}\cup \{(t,s): t-s =0, t>0\}.
		\]
		Thus, the elementwise $1$-norm is semialgebraic.
		
		\item \textbf{the elastic net:}
		Note that the elastic net is the sum of the elementwise $1$-norm and the squared Frobenius norm. Thus, by (c2), (c3) and  \Cref{Lemm:semialgbraic}(3), the elastic net is semialgebraic.
		
		\item
		If $r_i$ or $s_i$ is the indicator function of nonnegative closed half space or a closed interval (box constraints), by \Cref{Lemm:semialgbraic}(1), any polyhedral set is semialgebraic such as the nonnegative orthant $\Rp^{p\times q}=\{\bW\in \RR^{p\times q}, \bW_{ij} \geq 0, \forall i,j\}$, and the closed interval. Thus, by \Cref{Lemm:semialgbraic}(4), $r_i$ or $s_i$ is semialgebraic in this case.
\hfill$\Box$
	\end{enumerate}

\end{proof}

\section{Proof of \Cref{Thm:BCD-ConvThm1}}
\label{sc:proof-main-theorem}

To prove \Cref{Thm:BCD-ConvThm1},
we first show that the Kurdyka-{\L}ojasiewicz (\KL) property holds for the considered DNN training models (see \Cref{Thm:KL-property}),
then establish the function value convergence of the BCD methods (see \Cref{Thm:BCD-ConvThm-value}), followed by establishing their global convergence as well as the $\calO(1/k)$ convergence rate to a critical point as shown in \Cref{Thm:BCD-ConvThm}.
Combining \Cref{Thm:KL-property}, \Cref{Thm:BCD-ConvThm-value,Thm:BCD-ConvThm} yields \Cref{Thm:BCD-ConvThm1}.

\subsection{The Kurdyka-{\L}ojasiewicz Property in Deep Learning}
\label{sc:KL}

Before giving the definition of the \KL property, we first introduce some notions and notations from variational analysis, which can be found in \citet{Rockafellar1998}.

The notion of subdifferential plays a central role in the following definition of the \KL property.
For each $\bx\in \dom(h)$, the \textit{Fr\'{e}chet subdifferential} of $h$ at $\bx$, written $\widehat{\partial}h(\bx)$, is the set of vectors $\bv\in \RR^p$ which satisfy
\[
\liminf_{\by\neq \bx, \by\rightarrow \bx} \ \frac{h(\by)-h(\bx)-\langle \bv,\by-\bx\rangle}{\|\bx-\by\|} \geq 0.
\]
When $\bx\notin \dom(h),$ we set $\widehat{\partial} h(\bx) = \varnothing.$
The \emph{limiting-subdifferential} (or simply \emph{subdifferential}) of $h$ introduced in \citet{Mordukhovich-2006}, written $\partial h(\bx)$ at $\bx\in \mathrm{dom}(h)$,  is defined by
\begin{align}
\label{Def:limiting-subdifferential}
\partial h(\bx) := \{\bv\in \RR^p: \exists \bx^k \to \bx,\; h(\bx^k)\to h(\bx), \; \bv^k \in \widehat{\partial} h(\bx^k) \to \bv\}.
\end{align}
%
A necessary (but not sufficient) condition for $\bx\in \RR^p$ to be a minimizer of $h$ is $\zero \in \partial h(\bx)$.
A point that satisfies this inclusion is called \textit{limiting-critical} or simply \textit{critical}.
The distance between a point $\bx$ to a subset ${\cal S}$ of $\RR^p$, written $\mathrm{dist}(\bx,\calS)$, is defined by $\mathrm{dist}(\bx,\calS) = \inf \{\|\bx-\bs\|: \bs\in \calS\}$, where $\|\cdot\|$ represents the Euclidean norm.

The \KL property \citep{Lojasiewicz-KL1963,Lojasiewicz-KL1993,Kurdyka-KL1998,Bolte-KL2007a,Bolte-KL2007} plays a central role in the convergence analysis of nonconvex algorithms \citep[see e.g.,][]{Attouch2013,Xu-Yin-BCD2013,Wang-ADMM2018}. The following definition is adopted from \citet{Bolte-KL2007a}.

\begin{definition}[Kurdyka-{\L}ojasiewicz property]
	\label{Def-KLProp}
	A function
	$h:\RR^p \rightarrow \RR\cup \{+\infty\}$
	is said to have the \textbf{Kurdyka-{\L}ojasiewicz (\KL) property} at $\bx^*\in\dom(\partial h)$ if there exist a neighborhood $U$ of $\bx^*$, a constant $\eta$, and a continuous concave function $\varphi(s) = cs^{1-\theta}$ for some $c>0$ and $\theta \in [0,1)$
	such that the Kurdyka-{\L}ojasiewicz inequality holds: For all $\bx \in U \cap \dom(\partial h)$ and $h(\bx^*) < h(\bx) < h(\bx^*)+\eta$,
	\begin{equation}
	\varphi'(h(\bx)-h(\bx^*)) \cdot\mathrm{dist}(\zero,\partial h(\bx))\geq 1, 	\label{KLIneq}
	\end{equation}
	where $\theta$ is called the \KL exponent of $h$ at $\bx^*$. Proper lower semi-continuous functions which satisfy the Kurdyka-{\L}ojasiewicz inequality at each point of $\dom(\partial h)$ are called \KL functions.
\end{definition}
Note that we have adopted in the definition of the \KL inequality \eqref{KLIneq} the following notational conventions: $0^0=1, \infty/\infty=0/0=0.$
Such property was firstly introduced by \citet{Lojasiewicz-KL1993} on real analytic functions \citep{Krantz2002-real-analytic} for $\theta \in \left[ \tfrac{1}{2},1\right) $, then was extended to functions defined on the o-minimal structure in \citet{Kurdyka-KL1998}, and later was extended to nonsmooth subanalytic functions in \citet{Bolte-KL2007a}.

By the definition of the \KL property, it means that the function under consideration is sharp up to a reparametrization \citep{Attouch2013}.
Particularly, when $h$ is smooth, finite-valued, and $h(\bx^*)=0$, the inequality \eqref{KLIneq} can be rewritten
\[
\|\nabla (\varphi \circ h)(\bx)\|\geq 1,
\]
for each convenient $\bx \in \RR^p$. This inequality may be interpreted as follows: up to the reparametrization of the values of $h$ via $\varphi$, we face a \textit{sharp function}. Since the function $\varphi$ is used here to turn a singular region---a region in which the gradients are arbitrarily small---into a regular region, i.e., a place where the gradients are bounded away from zero, it is called a \textit{desingularizing} function for $h$. For theoretical and geometrical developments concerning this inequality, see \citet{Bolte-KL2007}.
\KL functions include real analytic functions (see \Cref{Def:real-analytic}), semialgebraic functions (see \Cref{Def:semialgebraic}), tame functions defined in some o-minimal structures \citep{Kurdyka-KL1998}, continuous subanalytic functions \citep{Bolte-KL2007a,Bolte-KL2007} and locally strongly convex functions \citep{Xu-Yin-BCD2013}.


In the following, we establish the \KL properties\footnote{It should be pointed out that we need to use the vectorization of the matrix variables involved in ${\cal L}$, $\bcL$ and $\bcL_{\res}$ in order to adopt the existing definitions of \KL property, real analytic functions and semialgebraic functions. We still use the matrix notation for the simplicity of notation.} of the DNN training models with variable splitting, i.e., the functions  $\calL$ defined in \eqref{Eq:2-split-min} and $\overline{\calL}$ defined in \eqref{Eq:3-split-min}.

\begin{proposition}[\KL properties of deep learning]
\label{Thm:KL-property}
Suppose that \Cref{Assumption-model} hold. Then the functions $\calL$ defined in \eqref{Eq:2-split-min}, and $\bcL$ defined in \eqref{Eq:3-split-min} when restricted to any closed set are \KL functions.	
\end{proposition}

This proposition shows that most of the DNN training models with variable splitting have some \textit{nice} geometric properties, i.e., they are amenable to sharpness at each point in their domains.
In order to prove this theorem, we need the following lemmas.
The first lemma shows some important properties of real analytic functions.

\begin{lemma}[\citealp{Krantz2002-real-analytic}]
\label{Lemm:real-analytic}
The sums, products, and compositions of real analytic functions are real analytic functions.
\end{lemma}
Then we present some important properties of semialgebraic sets and mappings, which can be found in \citet{Bochnak-semialgebraic1998}.
\begin{lemma}
\label{Lemm:semialgbraic}
The following hold
\begin{enumerate}
\item[(1)]
The finite union, finite intersection, and complement of semialgebraic sets are semialgebraic.
The closure and the interior of a semialgebraic set are semialgebraic \citep[Proposition 2.2.2]{Bochnak-semialgebraic1998}.

\item[(2)] The composition $g\circ h$ of semialgebraic mappings $h:A\rightarrow B$ and $g:B\rightarrow C$ is semialgebraic
\citep[Proposition 2.2.6]{Bochnak-semialgebraic1998}.
		
\item[(3)] The sum of two semialgebraic functions is semialgebraic (can be referred to the proof of
		\citealt[Proposition 2.2.6]{Bochnak-semialgebraic1998}).
		
\item[(4)] The indicator function of a semialgebraic set is semialgebraic \citep{Bochnak-semialgebraic1998}.
\end{enumerate}
\end{lemma}

Since our proof involves the sum of real analytic functions and semialgebraic functions, we still need the following lemma, of which the claims can be found in or derived directly from \citet{Shiota1997-subanalytic}.
\begin{lemma}
	\label{Lemm:real-semialg-subanalytic}
	The following hold:
	\begin{enumerate}
		\item[(1)]
		Both real analytic functions and semialgebraic functions (mappings) are subanalytic \citep{Shiota1997-subanalytic}.
		
		\item[(2)] Let $f_1$ and $f_2$ are both subanalytic functions, then the sum of $f_1+f_2$ is a subanalytic function if at least one of them map a bounded set to a bounded set or if both of them are nonnegative \citep[p.43]{Shiota1997-subanalytic}.
	\end{enumerate}
\end{lemma}

Moreover, we still need the following important lemma from \citet{Bolte-KL2007a}, which shows that the subanalytic function is a \KL function.
\begin{lemma}[{\citealp[Theorem 3.1]{Bolte-KL2007a}}\,\!]
	\label{Lemm:subanalytic-kl}
	Let $h:\RR^p \rightarrow \RR\cup \{+\infty\}$ be a subanalytic function with closed domain, and assume that $h$ is continuous on its domain, then $h$ is a \KL function.
\end{lemma}



\begin{proof}[Proof of \Cref{Thm:KL-property}]
	We first verify the \KL property of $\overline{\calL}$, and then similarly show that of ${\cal L}$.
From \eqref{Eq:3-split-min},
\begin{multline*}
{\bcL}(\{\bW_i\}_{i=1}^N,\{\bV_i\}_{i=1}^N,\{\bU_i\}_{i=1}^N) \\
:= \calR_n(\bV_N; \bY) + \sum_{i=1}^N r_i(\bW_i) + \sum_{i=1}^N s_i(\bV_i)  + \frac{\gamma}{2}\sum_{i=1}^N \left[ \|\bV_i - \sigma_i(\bU_i)\|_F^2 + \|\bU_i - \bW_i \bV_{i-1}\|_F^2 \right],
\end{multline*}
which mainly includes the following types of functions, i.e.,
\begin{align*}
	\calR_n (\bV_N; \bY), \ r_i(\bW_i), \ s_i(\bV_i), \  \|\bV_i - \sigma_i(\bU_i)\|_F^2,\  \|\bU_i - \bW_i \bV_{i-1}\|_F^2.
\end{align*}
To verify the \KL property of the function $\bcL$, we consider the above functions one by one under the hypothesis of \Cref{Thm:KL-property}.
	
\textbf{On $\calR_n(\bV_N; \bY)$:}
Note that given the output data $\bY$, $\calR_n(\bV_N; \bY):= \frac{1}{n} \sum_{j=1}^n \ell((\bV_N)_{:j}, \by_j)$, where $\ell: \RR^{d_N} \times \RR^{d_N} \rightarrow \RR_{+}\cup \{0\}$ is some loss function.
If $\ell$ is real analytic (resp.~semialgebraic), then by \Cref{Lemm:real-analytic} (resp.~\Cref{Lemm:semialgbraic}(3)), $\calR_n(\bV_N; \bY)$ is real-analytic (resp.~semialgebraic).
	
\textbf{On $\|\bV_i - \sigma_i(\bU_i)\|_F^2$:}
Note that $\|\bV_i - \sigma_i(\bU_i)\|_F^2$ is a finite sum of simple functions of the form, $|v-\sigma_i(u)|^2$ for any $u,v \in \RR$. If $\sigma_i$ is real analytic (resp.~semialgebraic), then $v-\sigma_i(u)$ is real analytic (resp.~semialgebraic), and further by \Cref{Lemm:real-analytic} (resp.~\Cref{Lemm:semialgbraic}(2)), $|v-\sigma_i(u)|^2$ is also real analytic (resp.~semialgebraic) since $|v-\sigma_i(u)|^2$ can be viewed as the composition $g\circ h$ of these two functions where $g(t)=t^2$ and $h(u,v) = v-\sigma_i(u)$.
	
\textbf{On $\|\bU_i - \bW_i \bV_{i-1}\|_F^2$: }
Note that the function $\|\bU_i - \bW_i \bV_{i-1}\|_F^2$ is a polynomial function with the variables $\bU_i, \bW_i$ and $\bV_{i-1}$, and thus according to \citet{Krantz2002-real-analytic} and \citet{Bochnak-semialgebraic1998}, it is both real analytic and semialgebraic.
	
\textbf{On $r_i(\bW_i)$, $s_i(\bV_i)$:}
All $r_i$'s and $s_i$'s are real analytic or semialgebraic by the hypothesis of \Cref{Thm:KL-property}.
	
Since each part of the function $\overline{\calL}$ is either real analytic or semialgebraic, then by \Cref{Lemm:real-semialg-subanalytic}, $\overline{\calL}$ is a subanalytic function. Furthermore, by the continuity hypothesis of \Cref{Thm:KL-property}, $\overline{\calL}$ is continuous in its domain. Therefore, $\overline{\calL}$ is a \KL function according to \Cref{Lemm:subanalytic-kl}.
	
Similarly, we can verify the \KL property of $\calL$ by checking each part is either real analytic or semialgebraic.
The major task is to check the \KL properties of the functions $\|\bV_i - \sigma_i(\bW_i \bV_{i-1})\|_F^2$ ($i=1,\ldots,N$). This reduces to check the function $h: \RR\times \RR^{d_{i-1}}\times \RR^{d_{i-1}} \rightarrow \RR$, $h(u, \bv, \bw):= |u-\sigma_i(\left\langle \bw, \bv\right\rangle )|^2$. Similar to the case $|v-\sigma_i(u)|^2$ for any $u,v \in \RR$ in $\bcL$, $h$ is real analytic (resp.~semialgebraic) if $\sigma_i$ is real analytic (resp.~semialgebraic) by  \Cref{Lemm:real-analytic} (resp.~\Cref{Lemm:semialgbraic}(2)).
As a consequence, each part of the function $\calL$ is either real analytic or semialgebraic, so ${\cal L}$ is a subanalytic function, and further by the continuity hypothesis of \Cref{Thm:KL-property}, ${\cal L}$ is a \KL function according to \Cref{Lemm:subanalytic-kl}.
This completes the proof. \hfill$\Box$

\end{proof}



\subsection{Value convergence of BCD}

We show the value convergence of both algorithms as follows.

\begin{theorem}
\label{Thm:BCD-ConvThm-value}
Let $\{{\cal Q}^k:=\left(\{\bW_i^k\}_{i=1}^N, \{\bV_i^k\}_{i=1}^N\right)\}_{k\in\NN}$ and $\{{\cal P}^k:=\left(\{\bW_i^k\}_{i=1}^N, \{\bV_i^k\}_{i=1}^N, \{\bU_i^k\}_{i=1}^N\right)\}_{k\in\NN}$ be the sequences generated by \Cref{alg:BCD-2-split,alg:BCD-3-split}, respectively. Under \Cref{Assumption-model} and finite initializations ${\cal Q}^0$ and ${\cal P}^0$, then for any positive $\alpha$ and $\gamma$, $\{{\cal L}({\cal Q}^k)\}_{k\in\NN}$ (resp.~$\{\bcL({\cal P}^k)\}_{k\in\NN}$) is nonincreasing and converges to some finite ${\cal L}^*$ (resp.~$\overline{\calL}^*$).
\end{theorem}

In order to prove \Cref{Thm:BCD-ConvThm-value}, we first show the convergence of \Cref{alg:BCD-3-split} and then show that of \Cref{alg:BCD-2-split} similarly. We restate \Cref{Lemm:BCD-suff-desc} precisely as follows.

\begin{lemma}[Restate of \Cref{Lemm:BCD-suff-desc}]
	\label{Lemm:BCD-suff-desc1}
	Let $\{{\cal P}^k\}_{k\in\NN}$ be a sequence generated by the BCD method (\Cref{alg:BCD-3-split}),
	then
	\begin{align}
	\label{Eq:BCD-suff-desc}
	\bcL({\cal P}^k)  \leq \overline{\calL}(\calP^{k-1}) - a\|\calP^k - \calP^{k-1}\|_F^2,
	\end{align}
	where
	\begin{align}
	\label{Eq:def-a}
	a:=\min\left\{\frac{\alpha}{2},\frac{\gamma}{2}\right\}
	\end{align}
	for the case that $\bV_N$ is updated via the proximal strategy, or
	\begin{align}
	\label{Eq:def-a1}
	a:=\min\left\{\frac{\alpha}{2},\frac{\gamma}{2}, \alpha + \frac{\gamma-L_R}{2}\right\}
	\end{align}
	for the case that $\bV_N$ is update via the prox-linear strategy.
\end{lemma}

According to \Cref{alg:BCD-3-split}, the decreasing property of the sequence $\{\bcL(\calP^k)\}_{k\in\NN}$ is obvious. However, establishing the sufficient descent inequality \eqref{Eq:BCD-suff-desc} for the sequence $\{\bcL(\calP^k)\}_{k\in\NN}$ is nontrivial. To achieve this, we should take advantage of the specific update strategies and also the form of $\bcL$ as shown in the following proofs.

\begin{proof}
	The descent quantity in \eqref{Eq:BCD-suff-desc} can be developed via considering the descent quantity along the update of each block variable.
	From \Cref{alg:BCD-3-split}, each block variable is updated either by the proximal strategy with parameter $\alpha/2$ (say, updates of $\bV_N^k, \{\bU_i^k\}_{i=1}^{N-1}, \{\bW_i^k\}_{i=1}^{N}$-blocks in \Cref{alg:BCD-3-split}) or by minimizing a strongly convex function\footnote{The function $h$ is called a strongly convex function with parameter $\gamma>0$  if $h(u) \geq h(v) + \langle \nabla h(v), u-v\rangle + \frac{\gamma}{2} \|u-v\|^2$.} with parameter $\gamma>0$ (say, updates of $\{\bV_i^k\}_{i=1}^{N-1}, \bU_N^k$-blocks in \Cref{alg:BCD-3-split}),
	we will consider both cases one by one.
	
	\textbf{(a) Proximal update case:}
	In this case, we take the $\bW_i^k$-update case for example.
	By \Cref{alg:BCD-3-split}, $\bW_i^k$ is updated according to the following
	\begin{equation}
	\label{Eq:Wi-update} \bW_i^k \leftarrow \argmin_{\bW_i} \left\{r_i(\bW_i) + \frac{\gamma}{2} \|\bU_i^k - \bW_i \bV_{i-1}^{k-1}\|_F^2 + \frac{\alpha}{2} \|\bW_i - \bW_i^{k-1}\|_F^2 \right\}.
	\end{equation}
	Let $h^k(\bW_i) = r_i(\bW_i) + \frac{\gamma}{2}\|\bU_i^k - \bW_i \bV_{i-1}^{k-1}\|_F^2$ and $\bar{h}^k(\bW_i)= r_i(\bW_i) + \frac{\gamma}{2}\|\bU_i^k - \bW_i \bV_{i-1}^{k-1}\|_F^2 + \frac{\alpha}{2}\|\bW_i-\bW_i^{k-1}\|_F^2$.
	By the optimality of $\bW_i^k$, the following holds
	\begin{align*}
	\bar{h}^k(\bW_i^{k-1}) \geq \bar{h}^k(\bW_i^k),
	\end{align*}
	which implies
	\begin{align}
	\label{Eq:descent-Wi-h}
	h^k(\bW_i^{k-1}) \geq h^k(\bW_i^k) + \frac{\alpha}{2} \|\bW_i^k - \bW_i^{k-1}\|_F^2.
	\end{align}
	Note that the $\bW_i^k$-update \eqref{Eq:Wi-update} is equivalent to the following original proximal BCD update, i.e.,
	\begin{align*}
	\bW_i^k \leftarrow \argmin_{\bW_i} \overline{\calL}(\bW_{<i}^{k-1}, \bW_i, \bW_{>i}^k, \bV_{<i}^{k-1}, \bV_i^k, \bV_{>i}^k, \bU_{<i}^{k-1}, \bU_i^k, \bU_{>i}^k) + \frac{\alpha}{2}\|\bW_i^k - \bW_i^{k-1}\|_F^2,
	\end{align*}
	where $\bW_{<i}:= (\bW_1, \bW_2, \ldots, \bW_{i-1})$, $\bW_{>i}:= (\bW_{i+1}, \bW_{i+1}, \ldots, \bW_N)$, and $\bV_{<i}, \bV_{>i}, \bU_{<i}, \bU_{>i}$ are defined similarly.
	Thus, by \eqref{Eq:descent-Wi-h}, we establish the descent part along the $\bW_i$-update ($i=1,\ldots, N-1$), i.e.,
	\begin{multline}
	\label{Eq:descent-Wi}
	\overline{\calL}(\bW_{<i}^{k-1}, {\color{red} \bW_i^{k-1}}, \bW_{>i}^k, \bV_{<i}^{k-1}, \bV_i^k, \bV_{>i}^k, \bU_{<i}^{k-1}, \bU_i^k, \bU_{>i}^k) \\
	\geq \overline{\calL}(\bW_{<i}^{k-1}, {\color{red} \bW_i^k}, \bW_{>i}^k, \bV_{<i}^{k-1}, \bV_i^k, \bV_{>i}^k, \bU_{<i}^{k-1}, \bU_i^k, \bU_{>i}^k)  + \frac{\alpha}{2} \|\bW_i^k - \bW_i^{k-1}\|_F^2.
	\end{multline}
	Similarly, we can establish the similar descent estimates of \eqref{Eq:descent-Wi} for the other blocks using the proximal updates including $\bV_N^k, \{\bU_i^k\}_{i=1}^{N-1}$ and $\bW_N^k$ blocks.
	
	Specifically, for the $\bV_N^k$-block, the following holds
	\begin{equation}
	\label{Eq:descent-VN}
	\overline{\calL}(\{\bW_{i}^{k-1}\}_{i=1}^N, \bV_{i<N}^{k-1}, {\color{red} \bV_N^{k-1}}, \{\bU_{i}^{k-1}\}_{i=1}^N)
	\geq \overline{\calL}(\{\bW_{i}^{k-1}\}_{i=1}^N, \bV_{i<N}^{k-1}, {\color{red} \bV_N^{k}}, \{\bU_{i}^{k-1}\}_{i=1}^N)  + \frac{\alpha}{2} \|\bV_N^k - \bV_N^{k-1}\|_F^2.
	\end{equation}
	For the $\{\bU_i^k\}$-block, $i=1,\ldots, N-1$, the following holds
	\begin{multline}
	\label{Eq:descent-Ui}
	\overline{\calL}(\bW_{<i}^{k-1}, \bW_i^{k-1}, \bW_{>i}^k, \bV_{<i}^{k-1}, \bV_i^k, \bV_{>i}^k, \bU_{<i}^{k-1}, {\color{red} \bU_i^{k-1}}, \bU_{>i}^k) \\
	\geq \overline{\calL}(\bW_{<i}^{k-1}, \bW_i^{k-1}, \bW_{>i}^k, \bV_{<i}^{k-1}, \bV_i^k, \bV_{>i}^k, \bU_{<i}^{k-1}, {\color{red} \bU_i^k}, \bU_{>i}^k)  + \frac{\alpha+\gamma}{2} \|\bU_i^k - \bU_i^{k-1}\|_F^2.
	\end{multline}
	For the $\bW_N^k$-block, the following holds
	\begin{equation}
	\label{Eq:descent-WN}
	\overline{\calL}(\bW_{i<N}^{k-1}, {\color{red} \bW_N^{k-1}}, \bV^{k-1}_{i<N}, \bV_N^k, \bU^{k-1}_{i<N}, \bU_N^{k}) \geq \overline{\calL}(\bW_{i<N}^{k-1}, {\color{red} \bW_N^{k}}, \bV^{k-1}_{i<N}, \bV_N^k, \bU^{k-1}_{i<N}, \bU_N^{k})  + \frac{\alpha}{2} \|\bW_N^k - \bW_N^{k-1}\|_F^2.
	\end{equation}
	
	\textbf{(b) Minimization of a strongly convex case:}
	In this case, we take $\bV_i^k$-update case for example.
	From \Cref{alg:BCD-3-split}, $\bV_i^k$ is updated according to the following
	\begin{align}
	\label{Eq:Vi-update}
	\bV_i^k \leftarrow \argmin_{\bV_i} \ \left\{s_i(\bV_i) + \frac{\gamma}{2} \|\bV_i - \sigma_i(\bU_i^{k-1})\|_F^2 + \frac{\gamma}{2}\|\bU_{i+1}^k - \bW_{i+1}^k \bV_i\|_F^2 \right\}.
	\end{align}
	Let $h^k(\bV_i) = s_i(\bV_i) + \frac{\gamma}{2} \|\bV_i - \sigma_i(\bU_i^{k-1})\|_F^2 + \frac{\gamma}{2}\|\bU_{i+1}^k - \bW_{i+1}^k \bV_i\|_F^2$. By the convexity of $s_i$, the function $h^k(\bV_i)$ is a strongly convex function with modulus no less than $\gamma$.
	By the optimality of $\bV_i^k$, the following holds
	\begin{align}
	\label{Eq:descent-Vi-h}
	h^k(\bV_i^{k-1}) \geq h^k(\bV_i^k) + \frac{\gamma}{2}\|\bV_i^k - \bV_i^{k-1}\|_F^2.
	\end{align}
	Noting the relation between $h^k(\bV_i)$ and $\overline{\calL}(\bW_{<i}^{k-1}, \bW_i^{k-1}, \bW_{>i}^k, \bV_{<i}^{k-1}, \bV_i, \bV_{>i}^k, \bU_{<i}^{k-1}, \bU_i^{k-1}, \bU_{>i}^k)$, and by \eqref{Eq:descent-Vi-h}, it yields for $i=1,\ldots,N-1$,
	\begin{multline}
	\label{Eq:descent-Vi}
	\overline{\calL}(\bW_{<i}^{k-1}, \bW_i^{k-1}, \bW_{>i}^k, \bV_{<i}^{k-1}, {\color{red} \bV_i^{k-1}}, \bV_{>i}^k, \bU_{<i}^{k-1}, \bU_i^{k-1}, \bU_{>i}^k) \\
	\geq \overline{\calL}(\bW_{<i}^{k-1}, \bW_i^{k-1}, \bW_{>i}^k, \bV_{<i}^{k-1}, {\color{red} \bV_i^k}, \bV_{>i}^k, \bU_{<i}^{k-1}, \bU_i^{k-1}, \bU_{>i}^k)  + \frac{\gamma}{2} \|\bV_i^k - \bV_i^{k-1}\|_F^2.
	\end{multline}
	Similarly, we can establish the similar descent estimates for the $\bU_N^k$-block, i.e.,
	\begin{equation}
	\label{Eq:descent-UN}
	\overline{\calL}(\bW_{<N}^{k-1}, \bW_N^{k-1}, \bV_{<N}^{k-1}, \bV_N^{k}, \bU_{<N}^{k-1}, {\color{red} \bU_N^{k-1}})
	\geq \overline{\calL}(\bW_{<N}^{k-1}, \bW_N^{k-1}, \bV_{<N}^{k-1}, \bV_N^{k}, \bU_{<N}^{k-1}, {\color{red} \bU_N^{k}})  + \gamma \|\bU_N^k - \bU_N^{k-1}\|_F^2.
	\end{equation}
	Summing \eqref{Eq:descent-Wi}--\eqref{Eq:descent-WN} and \eqref{Eq:descent-Vi}--\eqref{Eq:descent-UN} yields the descent inequality \eqref{Eq:BCD-suff-desc}.
	
	\textbf{(c) Prox-linear case for $\bV_N$, i.e., \eqref{Eq:VN-alg2}:} From \eqref{Eq:VN-alg2}, similarly, we let
	$h^k(\bV_N):=s_N(\bV_N)+\calR_n(\bV_N; \bY)+\frac{\gamma}{2}\|\bV_N - \bU_N^{k-1}\|_F^2$ and $\bar{h}^k(\bV_N) = s_N(\bV_N) + \calR_n(\bV_N^{k-1}; \bY)+\langle \nabla \calR_n(\bV_N^{k-1}; \bY), \bV_N-\bV_N^{k-1}\rangle + \frac{\alpha}{2}\|\bV_N-\bV_N^{k-1}\|_F^2 + \frac{\gamma}{2}\|\bV_N-\bU_N^{k-1}\|_F^2$.
	By the optimality of $\bV_N^k$ and the strong convexity of $\bar{h}^k(\bV_N)$ with modulus at least $\alpha + \gamma$, the following holds
	\[
	\bar{h}^k(\bV_N^{k-1}) \geq \bar{h}^k(\bV_N^{k}) + \frac{\alpha + \gamma}{2} \|\bV_N^{k}-\bV_N^{k-1}\|_F^2.
	\]
	After some simplifications and noting the relation between $h^k(\bV_N)$ and $\bar{h}^k(\bV_N)$, we have
	\begin{multline}
	h^k(\bV_N^{k-1})
	\geq h^k(\bV_N^k) - \left(\calR_n(\bV_N^k; \bY)- \calR_n(\bV_N^{k-1}; \bY) - \langle \nabla \calR_n(\bV_N^{k-1}; \bY), \bV_N^k - \bV_N^{k-1} \rangle\right)
	\\
	+\left( \alpha_N + \frac{\gamma}{2} \right)\|\bV_N^k - \bV_N^{k-1}\|_F^2 \geq h^k(\bV_N^k) + \left( \alpha + \frac{\gamma-L_R}{2}\right) \|\bV_N^k - \bV_N^{k-1}\|_F^2, \label{Eq:descent-VN-2}
	\end{multline}
	where the last inequality holds for the $L_R$-Lipschitz continuity of $\nabla \calR_n$, i.e., the following inequality by \citet{Nesterov-cvxopt-2018},
	\[
	\calR_n(\bV_N^k; \bY) \leq  \calR_n(\bV_N^{k-1}; \bY) + \langle \nabla \calR_n(\bV_N^{k-1}; \bY), \bV_N^k - \bV_N^{k-1} \rangle  + \frac{L_R}{2}\|\bV_N^k - \bV_N^{k-1}\|_F^2.
	\]
	
	Summing \eqref{Eq:descent-Wi}--\eqref{Eq:descent-WN}, \eqref{Eq:descent-Vi}--\eqref{Eq:descent-UN} and \eqref{Eq:descent-VN} yields the descent inequality \eqref{Eq:BCD-suff-desc}.
	\hfill$\Box$
\end{proof}

\begin{proof}[Proof of \Cref{Thm:BCD-ConvThm-value}]
	By \eqref{Eq:BCD-suff-desc}, $\bcL(\calP^k)$ is monotonically nonincreasing and lower bounded by $0$ since each term of $\bcL$ is nonnegative, thus, $\bcL(\calP^k)$ converges to some nonnegative, finite $\bcL^*$. Similarly, we can show the claims in \Cref{Thm:BCD-ConvThm-value} holds for \Cref{alg:BCD-2-split}. \hfill$\Box$
\end{proof}

Based on \Cref{Lemm:BCD-suff-desc1}, we can obtain the following corollary.
\begin{corollary}[Square summable]
	\label{Coro:square-summable}
	The following hold:
	\begin{enumerate}
		\item[(a)] $\sum_{k=1}^{\infty} \|{\calP}^k - {\calP}^{k-1}\|_F^{2} <\infty $,
        \item[(b)] $\frac{1}{K}\sum_{k=1}^{K} \|{\calP}^k - {\calP}^{k-1}\|_F^{2} \rightarrow 0$ at the rate of ${\cal O}(1/K)$, and
		\item[(c)] $\|{\calP}^k - {\calP}^{k-1}\|_F \to 0$ as $k\to \infty,$
	\end{enumerate}
\end{corollary}

\begin{proof}
	Summing \eqref{Eq:BCD-suff-desc} over $k$ from $1$ to $\infty$ yields
	\[
	\sum_{k=1}^{\infty} \|{\calP}^k - {\calP}^{k-1}\|_F^{2} \leq \overline{\calL}(\calP^0)<\infty,
	\]
	which directly implies $\|{\calP}^k - {\calP}^{k-1}\|_F \to 0$ as $k\to \infty.$
    Similarly, summing \eqref{Eq:BCD-suff-desc} over $k$ from $1$ to $K$ yields
	\[
	\frac{1}{K}\sum_{k=1}^{K} \|{\calP}^k - {\calP}^{k-1}\|_F^{2} \leq \frac{1}{K}\overline{\calL}(\calP^0),
	\]
    which implies claim (b) of this corollary.
    \hfill$\Box$
\end{proof}

\subsection{Global convergence of BCD}

\Cref{Thm:BCD-ConvThm-value} implies that the quality of the generated sequence is gradually improving during the iterative procedure in the sense of the descent of the objective value, and eventually achieves {some level of objective value, then keeps stable.
However, the convergence of the generated sequence $\{{\cal Q}^k\}_{k\in\NN}$ (resp.~$\{{\cal P}^k\}_{k\in\NN}$) itself is still unclear. In the following, we will show that under some natural conditions, the whole sequence converges to some critical point of the objective, and further if the initial point is sufficiently close to some global minimum, then the generated sequence can converge to this global minimum.

\begin{theorem}[Global convergence and rate]
	\label{Thm:BCD-ConvThm}
Under the assumptions of \Cref{Thm:BCD-ConvThm1}, the following hold
\begin{enumerate}[label=(\alph*)]
\item
$\{{\cal Q}^k\}_{k\in\NN}$ (resp.~$\{{\cal P}^k\}_{k\in\NN}$) converges to a critical point of ${\cal L}$ (resp.~$\overline{\calL}$).
	
\item
If further the initialization $\calQ^0$ (resp.~$\calP^0$) is sufficiently close to some global minimum $\calQ^*$ of $\cal L$ (resp.~$\calP^*$ of $\bcL$), then $\calQ^k$ (resp.~$\calP^k$) converges to $\calQ^*$ (resp.~$\calP^*$).
	
\item
Let $\theta$ be the \KL exponent of $\cal L$ (resp.~$\bcL$) at $\calQ^*$ (resp.~$\calP^*$). There hold: \textbf{(a)} if $\theta = 0$, then $\{\calQ^k\}_{k\in\NN}$ converges in a finite number of steps; \textbf{(b)} if $\theta \in \left( 0,\frac{1}{2}\right] $, then $\|\calQ^k - \calQ^*\|_F \leq C\tau^k$ for all $k\geq k_0$, for certain $k_0>0, C>0, \tau \in (0,1)$; and \textbf{(c)} if $\theta \in (\frac{1}{2},1)$, then $\|\calQ^k-\calQ^*\|_F \leq Ck^{-\frac{1-\theta}{2\theta-1}}$ for $k\geq k_0$, for certain $k_0>0, C>0$. The same claims hold for the sequence $\{\calP^k\}$.

\item
$\frac{1}{K}\sum_{k=1}^K \|\bg^k\|_F^2 \rightarrow 0$ at the rate ${\cal O}(1/K)$ where $\bg^k \in \partial {\cal L}({\cal Q}^k)$.
Similarly, $\frac{1}{K}\sum_{k=1}^K \|{\bar{\bg}}^k\|_F^2 \rightarrow 0$ at the rate ${\cal O}(1/K)$ where ${\bar{\bg}}^k \in \partial {\bcL}({\calQ}^k)$.
\end{enumerate}
\end{theorem}

Our proof is mainly based on the \textit{Kurdyka-{\L}ojasiewicz} framework established in \citet{Attouch2013} (some other pioneer work can be also found in \citealp{Attouch2010}). According to \citet{Attouch2013}, three key conditions including the \textit{sufficient decrease condition}, the \textit{relative error condition} and the \textit{continuity condition}, together with the \KL property at some limiting point are required to establish the global convergence of a descent algorithm from the subsequence convergence, where the sufficient decrease condition and the \KL property restricted to any closed set have been established in \Cref{Lemm:BCD-suff-desc1} and \Cref{Thm:KL-property}, respectively. The relative error condition is developed in \Cref{Lemm:BCD-grad-bound1}, while the continuity condition holds naturally due to the continuity assumption. Particularly, the closed set assumption in \Cref{Thm:KL-property} can be satisfied naturally by the boundedness of the sequence as well as its limiting points, which is yielded by \Cref{Lemm:BCD-suff-desc1} and the coerciveness assumption. By exploiting the boundedness, we only need to consider the \KL property of the objective function restricted to a large bounded, closed set including all limiting points, instead of the total real-valued set.
In the following, we first prove \Cref{Thm:BCD-ConvThm} under the subsequence convergence assumption, i.e., condition (a) of this theorem, and then show that both condition (b) and condition (c) can imply the boundedness of the sequence (see \Cref{Lemm:boundednes}), and thus the subsequence convergence as required in condition (a). The rate of convergence results follow the same argument as in the proof of \citet[Theorem 2]{Attouch2009}.



\subsubsection{Establishing relative error condition}
\label{app:estab-relative-error}

We restate \Cref{Lemm:BCD-grad-bound} precisely as follows.

\begin{lemma}[Restatement of \Cref{Lemm:BCD-grad-bound}]
	\label{Lemm:BCD-grad-bound1}
	Under the conditions of \Cref{Thm:BCD-ConvThm},
	let $\cal B$ be an upper bound of $\calP^{k-1}$ and $\calP^k$ for any positive integer $k$,
	$L_\calB$ be a uniform Lipschitz constant of $\sigma_i$ on the bounded set $\{\calP: \|\calP\|_F \leq \calB\}$, and
	\begin{align}
	\label{Eq:def-b}
	b:= \max\{\gamma, \alpha+\gamma \calB, \alpha + \gamma L_\calB, \gamma \calB + 2\gamma \calB^2, 2\gamma \calB + \gamma \calB^2\},
	\end{align}
	(or, for the prox-linear case, $b:= \max\{\gamma, L_R+\alpha+\gamma \calB, \alpha + \gamma L_\calB, \gamma \calB + 2\gamma \calB^2, 2\gamma \calB + \gamma \calB^2\}$),
	then for any positive integer $k$, there holds,
	\begin{align}
	\label{Eq:BCD-grad-bound}
	\mathrm{dist}(\zero,\partial \overline{\calL}({\cal P}^k))
	\leq b \sum_{i=1}^N [\|\bW_i^k - \bW_i^{k-1}\|_F + \|\bV_i^k - \bV_i^{k-1}\|_F + \|\bU_i^k - \bU_i^{k-1}\|_F ]
	\leq \bar{b} \|\calP^k - \calP^{k-1}\|_F,
	\end{align}
	where $\bar{b}:=b \sqrt{3N}$, $\mathrm{dist}(\zero,{\cal S}):= \inf_{\bs\in {\cal S}} \|\bs\|_F$ for a set ${\cal S}$, and
	\begin{align*}
	\partial \overline{\calL}({\cal P}^k)
	:= (\{\partial_{\bW_i}\overline{\calL}\}_{i=1}^N, \{\partial_{\bV_i} \overline{\calL}\}_{i=1}^N, \{\partial_{\bU_i} \overline{\calL} \}_{i=1}^N)(\calP^k).
	\end{align*}
\end{lemma}
\begin{proof}
	The inequality \eqref{Eq:BCD-grad-bound} is established via bounding each term of $\partial \overline{\calL}({\cal P}^k)$. By the optimality conditions of all updates in \Cref{alg:BCD-3-split}, the following hold
	\begin{align*}
	\zero & \in \partial s_N(\bV_N^k) + \partial \calR_n(\bV_N^k; \bY) + \gamma (\bV_N^k - \bU_N^{k-1}) + \alpha (\bV_N^k - \bV_N^{k-1}), \\
	& (\text{or for prox-linear}, 0 \in \partial s_N(\bV_N^k) + \nabla \calR_n(\bV_N^{k-1}; \bY) + \gamma (\bV_N^k - \bU_N^{k-1}) + \alpha (\bV_N^k - \bV_N^{k-1}),) \\
	 \zero &= \gamma (\bU_N^k - \bV_N^k)+ \gamma (\bU_N^k - \bW_N^{k-1} \bV_{N-1}^{k-1}), \\
	 \zero &\in \partial r_N(\bW_N^k) + \gamma (\bW_N^k \bV_{N-1}^{k-1} - \bU_N^k){\bV_{N-1}^{k-1}}^\top + \alpha (\bW_N^k - \bW_N^{k-1}),
	 \end{align*}
	for $i= N-1,\ldots, 1$,
	\begin{align*}
	 \zero &\in \partial s_i(\bV_i^k)+ \gamma (\bV_i^k - \sigma_i(\bU_i^{k-1})) + \gamma {\bW_{i+1}^k}^\top (\bW_{i+1}^k \bV_i^k - \bU_{i+1}^k), \\
	 \zero &\in  \gamma [(\sigma_i(\bU_i^k) - \bV_i^k)\odot \partial \sigma_i(\bU_i^k)] + \gamma (\bU_i^k - \bW_i^{k-1}\bV_{i-1}^{k-1}) + \alpha(\bU_i^k - \bU_i^{k-1}), \\
	 \zero &\in \partial r_i(\bW_i^k) + \gamma (\bW_i^k \bV_{i-1}^{k-1}- \bU_i^k){\bV_{i-1}^{k-1}}^\top  + \alpha (\bW_i^k - \bW_i^{k-1}),
	\end{align*}
	where $\bV_0^k \equiv \bV_0 = \bX$ for all $k$, and $\odot$ is the Hadamard product.
	By the above relations, we have
	\begin{align*}
	&-\alpha (\bV_N^k - \bV_N^{k-1}) -\gamma (\bU_N^k - \bU_N^{k-1})
	\in \partial s_N(\bV_N^k)+\partial \calR_n(\bV_N^k; \bY) + \gamma (\bV_N^k - \bU_N^{k})
	= \partial_{\bV_N} \overline{\calL}(\calP^k),\\
	&\left(\text{or}, (\nabla \calR_n(\bV_N^{k}; \bY)-\nabla \calR_n(\bV_N^{k-1}; \bY)) -\alpha (\bV_N^k - \bV_N^{k-1}) -\gamma (\bU_N^k - \bU_N^{k-1}) \in \partial_{\bV_N} \overline{\calL}(\calP^k) \right)\\
	&-\gamma (\bW_N^k - \bW_N^{k-1})\bV_{N-1}^k - \gamma \bW_N^{k-1}(\bV_{N-1}^k - \bV_{N-1}^{k-1})
	= \gamma (\bU_N^k - \bV_N^k)+ \gamma (\bU_N^k - \bW_N^{k} \bV_{N-1}^{k})
	= \partial_{\bU_N} \overline{\calL}(\calP^k),\\
	&\gamma {\bW_N^k}\left[ \bV_{N-1}^k ({\bV_{N-1}^k} - {\bV_{N-1}^{k-1}})^\top  + (\bV_{N-1}^k - \bV_{N-1}^{k-1}){\bV_{N-1}^{k-1}}^\top \right]
	- \gamma \bU_N^k(\bV_N^k - \bV_N^{k-1})^\top   - \alpha (\bW_N^k - \bW_N^{k-1})\\
	&\qquad\qquad\qquad\qquad\qquad\qquad\qquad\qquad\qquad\qquad\quad\in \partial r_N(\bW_N^k) + \gamma (\bW_N^k \bV_{N-1}^k - \bU_N^k){\bV_{N-1}^k}^\top
	= \partial_{\bW_N} \overline{\calL}(\calP^k),
	\end{align*}
	for $i= N-1,\ldots, 1$,
	\begin{align*}
    &-\gamma (\sigma_i(\bU_i^k) - \sigma_i(\bU_i^{k-1}))
	\in \partial s_i(\bV_i^k) + \gamma (\bV_i^k - \sigma_i(\bU_i^{k})) + \gamma {\bW_{i+1}^k}^\top  (\bW_{i+1}^k \bV_i^k - \bU_{i+1}^k)
	= \partial_{\bV_i} \overline{\calL}(\calP^k),\\
	&-\gamma \bW_i^{k-1}(\bV_{i-1}^k - \bV_{i-1}^{k-1}) -\gamma (\bW_i^k - \bW_i^{k-1})\bV_{i-1}^k
	-\alpha(\bU_i^k -\bU_i^{k-1}) \\
	&\qquad\qquad\qquad\qquad\qquad\qquad\qquad\qquad\qquad\in  \gamma [(\sigma_i(\bU_i^k) -\bV_i^k)\odot \partial \sigma_i(\bU_i^k)] + \gamma (\bU_i^k - \bW_i^{k}\bV_{i-1}^{k})
	= \partial_{\bU_i} \overline{\calL}(\calP^k),\\
	&\gamma {\bW_i^k}\left[ \bV_{i-1}^k (\bV_{i-1}^k - \bV_{i-1}^{k-1})^\top  + (\bV_{i-1}^k - \bV_{i-1}^{k-1}){\bV_{i-1}^{k-1}}^\top \right]
	- \gamma \bU_i^k(\bV_{i-1}^k - \bV_{i-1}^{k-1})^\top   - \alpha (\bW_i^k - \bW_i^{k-1})\nonumber\\
	&\qquad\qquad\qquad\qquad\qquad\qquad\qquad\qquad\qquad\qquad\qquad\qquad\in \partial r_i(\bW_i^k) + \gamma (\bW_i^k \bV_{i-1}^k - \bU_i^k){\bV_{i-1}^k}^\top
	= \partial_{\bW_i} \overline{\calL}(\calP^k).
	\end{align*}
	Based on the above relations, and by the Lipschitz continuity of the activation function on the bounded set $\{\calP: \|\calP\|_F \leq \calB\}$ and the bounded assumption of both $\calP^{k-1}$ and $\calP^k$, we have
	\begin{align*}
	\|\calG_{\bV_N}^k\|_F & \leq \alpha \|\bV_N^k - \bV_{N}^{k-1}\|_F + \gamma \|\bU_N^k  - \bU_N^{k-1}\|_F, && \calG_{\bV_N}^k \in \partial_{\bV_N} \bcL(\calP^k),\\
	&(\text{or}, \|\calG_{\bV_N}^k\|_F  \leq (L_R+\alpha) \|\bV_N^k - \bV_{N}^{k-1}\|_F + \gamma \|\bU_N^k  - \bU_N^{k-1}\|_F)\\
	\|\calG_{\bU_N}^k\|_F & \leq \gamma \calB \|\bW_N^k - \bW_N^{k-1}\|_F + \gamma \calB \|\bV_{N-1}^k - \bV_{N-1}^{k-1}\|_F, && \calG_{\bU_N}^k \in \partial_{\bU_N} \bcL(\calP^k),\\
	\|\calG_{\bW_N}^k\|_F & \leq 2\gamma \calB^2 \|\bV_{N-1}^k - \bV_{N-1}^{k-1}\|_F + \gamma \calB \|\bV_{N}^k - \bV_{N}^{k-1}\|_F + \alpha \|\bW_{N}^k - \bW_{N}^{k-1}\|_F, && \calG_{\bW_N}^k \in \partial_{\bW_N} \bcL(\calP^k),
	\end{align*}
	and for $i=N-1,\ldots,1$,
	\begin{align*}
	\|\calG_{\bV_i}^k\|_F & \leq \gamma L_\calB \|\bU_i^k - \bU_i^{k-1}\|_F, && \calG_{\bV_i}^k \in \partial_{\bV_i} \bcL(\calP^k),\\
	\|\calG_{\bU_i}^k\|_F & \leq \gamma\calB \|\bV_{i-1}^k - \bV_{i-1}^{k-1}\|_F + \gamma \calB \|\bW_i^k - \bW_i^{k-1}\|_F + \alpha \|\bU_i^k - \bU_i^{k-1}\|_F, && \calG_{\bU_i}^k \in \partial_{\bU_i} \bcL(\calP^k),\\
	\|\calG_{\bW_i}^k\|_F & \leq (\gamma \calB^2+\gamma \calB) \|\bV_{i-1}^k - \bV_{i-1}^{k-1}\|_F + \alpha \|\bW_{i}^k - \bW_{i}^{k-1}\|_F, && \calG_{\bW_i}^k \in \partial_{\bW_i} \bcL(\calP^k).
	\end{align*}
	Summing the above inequalities and after some simplifications, we obtain \eqref{Eq:BCD-grad-bound}.
	\hfill$\Box$
\end{proof}

\subsubsection{Proof of \Cref{Thm:BCD-ConvThm} under condition (a)}
\label{app:proof-thm5}
Based on \Cref{Thm:BCD-ConvThm-value} and under the hypothesis that $\bcL$ is continuous on its domain and there exists a convergent subsequence (i.e., condition (a)), the \textit{continuity condition} required in \citet{Attouch2013} holds naturally, i.e., there exists a subsequence $\{\calP^{k_j}\}_{j\in \mathbb{N}}$ and ${\cal P}^*$ such that
\begin{align}
\label{Eq:continuity-cond}
\calP^{k_j} \to \calP^* \quad \text{and} \quad \overline{\calL}(\calP^{k_j}) \to \overline{\calL}(\calP^*), \;\text{as}\;\; j\to \infty.
\end{align}
Based on \Cref{Lemm:BCD-suff-desc1,Lemm:BCD-grad-bound1}, and \eqref{Eq:continuity-cond}, we can justify the global convergence of $\calP^k$ stated in \Cref{Thm:BCD-ConvThm}, following the proof idea of \citet{Attouch2013}.
For the completeness of the proof, we still present the detailed proof as follows.

Before presenting the main proof, we establish a local convergence result of $\calP^k$, i.e., the convergence of $\calP^k$ when $\calP^0$ is sufficiently close to some point $\calP^*$. Specifically, let $(\varphi, \eta, U)$ be the associated parameters of the \KL property of $\bcL$ at $\calP^*$, where $\varphi$ is a continuous concave function, $\eta$ is a positive constant, and $U$ is a neighborhood of $\calP^*$.
Let $\rho$ be some constant such that
${\cal N}(\calP^*,\rho):= \{\calP: \|\calP - \calP^*\|_F \leq \rho\} \subset U$,
$\calB:=\rho + \|\calP^*\|_F$, and $L_\calB$ be the uniform Lipschitz constant for $\sigma_i$, $i=1,\ldots, N-1$, within ${\cal N}(\calP^*, \rho)$.
Assume that $\calP^0$ satisfies the following condition
\begin{align}
\label{Eq:initial-cond}
\frac{\bar{b}}{a}\varphi(\bcL(\calP^0) - \bcL(\calP^*)) + 3\sqrt{\frac{\bcL(\calP^0)}{a}} + \|\calP^0 - \calP^*\|_F < \rho,
\end{align}
where $\bar{b} = b\sqrt{3N}$, $b$ and $a$ are defined in \eqref{Eq:def-b} and \eqref{Eq:def-a}, respectively,

\begin{lemma}[Local convergence]
	\label{Lemm:local-finite-length}
	Under the conditions of \Cref{Thm:BCD-ConvThm}, suppose that $\calP^0$ satisfies the condition \eqref{Eq:initial-cond}, and $\bcL(\calP^k) > \bcL(\calP^*)$ for $k \in \mathbb{N}$, then
	\begin{align}
	\sum_{i=1}^k \|\calP^i - \calP^{i-1}\|_F &\leq 2\sqrt{\frac{\bcL(\calP^{0})}{a}}+\frac{\bar b}{a}\varphi(\bcL(\calP^{0})-\bcL(\calP^*)), \ \forall k\geq 1, \label{Eq:local-finite}\\
	\calP^k &\in {\cal N}(\calP^*, \rho), \quad \forall k\in \mathbb{N}. \label{Eq:bound}
	\end{align}
	As $k$ goes to infinity, \eqref{Eq:local-finite} yields
	\[
	\sum_{i=1}^\infty \|\calP^i - \calP^{i-1}\|_F < \infty,
	\]
	which implies the convergence of $\{\calP^k\}_{k\in\NN}$.
\end{lemma}
\begin{proof}
	We will prove $\calP^k \in {\cal N}(\calP^*,\rho)$ by induction on $k$.
	It is obvious that $\calP^0 \in {\cal N}(\calP^*,\rho)$. Thus, \eqref{Eq:bound} holds for $k=0$. For $k=1$, we have from \eqref{Eq:BCD-suff-desc} and the nonnegativeness of $\{\bcL(\calP^k)\}_{k\in\NN}$ that
	\[
	\bcL(\calP^0) \geq \bcL(\calP^0) - \bcL(\calP^1) \geq a \|\calP^0 - \calP^1\|_F^2,
	\]
	which implies $\|\calP^0 - \calP^1\|_F \leq \sqrt{\frac{\bcL(\calP^0)}{a}}$.
	Therefore,
	\[
	\|\calP^1 - \calP^*\|_F \leq \|\calP^0 - \calP^1\|_F + \|\calP^0 - \calP^*\|_F \leq \sqrt{\frac{\bcL(\calP^0)}{a}} + \|\calP^0 - \calP^*\|_F,
	\]
	which indicates $\calP^1 \in {\cal N}(\calP^*,\rho)$.
	
	Suppose that $\calP^k\in {\cal N}(\calP^*,\rho)$ for $0\leq k \leq K$. We proceed to show that $\calP^{K+1} \in {\cal N}(\calP^*,\rho)$.
	Since $\calP^k\in {\cal N}(\calP^*,\rho)$ for $0\leq k \leq K$, it implies that $\|\calP^k\|_F \leq \calB:= \rho + {\calP^*}$ for $0\leq k \leq K$.
	Thus, by \Cref{Lemm:BCD-grad-bound1}, for $1\leq k \leq K$,
	\[
	\mathrm{dist}(\zero,\partial \bcL(\calP^k)) \leq \bar{b} \|\calP^k - \calP^{k-1}\|_F,
	\]
	which together with the \KL inequality \eqref{KLIneq} yields
	\begin{align}
	\label{Eq:subgrad}
	\frac{1}{\varphi'(\bcL(\calP^k)-\bcL(\calP^*))} \leq \bar{b} \|\calP^k - \calP^{k-1}\|_F.
	\end{align}
	By \eqref{Eq:BCD-suff-desc}, the above inequality and the concavity of $\varphi$, for $k\geq 2$, the following holds
	\begin{align*}
	a\|\calP^k - \calP^{k-1}\|_F^2
	&\leq \bcL(\calP^{k-1}) - \bcL(\calP^{k}) = (\bcL(\calP^{k-1})-\bcL(\calP^*)) - (\bcL(\calP^{k})-\bcL(\calP^*)) \nonumber\\
	&\leq \frac{\varphi(\bcL(\calP^{k-1})-\bcL(\calP^*)) - \varphi(\bcL(\calP^{k})-\bcL(\calP^*))}{\varphi'(\bcL(\calP^{k-1}) - \bcL(\calP^*))} \nonumber\\
	&\leq {\bar b}\|\calP^{k-1} - \calP^{k-2}\|_F \cdot \left[\varphi(\bcL(\calP^{k-1})-\bcL(\calP^*)) - \varphi(\bcL(\calP^{k})-\bcL(\calP^*))\right],
	\end{align*}
	which implies
	\begin{align*}
	\|\calP^k - \calP^{k-1}\|_F^2
	\leq \|\calP^{k-1} - \calP^{k-2}\|_F \cdot \frac{\bar b}{a}\left[\varphi(\bcL(\calP^{k-1})-\bcL(\calP^*)) - \varphi(\bcL(\calP^{k})-\bcL(\calP^*))\right].
	\end{align*}
	Taking the square root on both sides and using the inequality $2\sqrt{\alpha \beta} \leq \alpha + \beta$, the above inequality implies
	\begin{align*}
	2\|\calP^k - \calP^{k-1}\|_F \leq \|\calP^{k-1} - \calP^{k-2}\|_F + \frac{\bar b}{a}\left[\varphi(\bcL(\calP^{k-1})-\bcL(\calP^*)) - \varphi(\bcL(\calP^{k})-\bcL(\calP^*))\right].
	\end{align*}
	Summing the above inequality over $k$ from $2$ to $K$ and adding $\|\calP^1 - \calP^{0}\|_F$ to both sides, it yields
	\begin{equation*}
	\|\calP^K - \calP^{K-1}\|_F+\sum_{k=1}^{K} \|\calP^k - \calP^{k-1}\|_F \nonumber
	\leq 2\|\calP^1 - \calP^{0}\|_F+\frac{\bar b}{a}\left[\varphi(\bcL(\calP^{0})-\bcL(\calP^*)) - \varphi(\bcL(\calP^{K})-\bcL(\calP^*))\right]
	\end{equation*}
	which implies
	\begin{equation}
	\label{Eq:finite-sum}
	\sum_{k=1}^{K} \|\calP^k - \calP^{k-1}\|_F \leq 2\sqrt{\frac{\bcL(\calP^{0})}{a}} + \frac{\bar b}{a}\varphi(\bcL(\calP^{0})-\bcL(\calP^*)),
	\end{equation}
	and further,
	\begin{align*}
	\|\calP^{K+1} - \calP^*\|_F
	&\leq \|\calP^{K+1} - \calP^K\|_F + \sum_{k=1}^K \|\calP^k - \calP^{k-1}\|_F + \|\calP^0 - \calP^*\|_F \\
	&\leq \sqrt{\frac{\bcL(\calP^{K})-\bcL(\calP^{K+1})}{a}} + 2\sqrt{\frac{\bcL(\calP^{0})}{a}}+ \frac{\bar b}{a}\varphi(\bcL(\calP^{0})-\bcL(\calP^*)) + \|\calP^0 - \calP^*\|_F\\
	&\leq 3\sqrt{\frac{\bcL(\calP^{0})}{a}}+ \frac{\bar b}{a}\varphi(\bcL(\calP^{0})-\bcL(\calP^*)) + \|\calP^0 - \calP^*\|_F < \rho,
	\end{align*}
	where the second inequality holds for \eqref{Eq:BCD-suff-desc} and \eqref{Eq:finite-sum}, the third inequality holds for $\bcL(\calP^{K})-\bcL(\calP^{K+1}) \leq \bcL(\calP^{K}) \leq \bcL(\calP^{0})$.
	Thus, $\calP^{K+1} \in {\cal N}(\calP^*,\rho)$. Therefore, we prove this lemma.
\hfill$\Box$
\end{proof}

\begin{proof}[Proof of \Cref{Thm:BCD-ConvThm}]
	We prove the whole sequence convergence stated in \Cref{Thm:BCD-ConvThm} according to the following two cases.
	
	\textbf{Case 1: $\bcL(\calP^{k_0}) = \bcL(\calP^*)$ at some $k_0$.} In this case, by \Cref{Lemm:BCD-suff-desc1}, $\calP^k = \calP^{k_0} = \calP^*$ holds for all $k\geq k_0$, which implies the convergence of $\calP^k$ to a limit point $\calP^*$.
	
	\textbf{Case 2: $\bcL(\calP^k)>\bcL(\calP^*)$ for all $k\in \mathbb{N}$.} In this case, since $\calP^*$ is a limit point and $\bcL(\calP^k) \rightarrow \bcL(\calP^*)$, by \Cref{Thm:BCD-ConvThm-value}, there must exist an integer $k_0$ such that $\calP^{k_0}$ is sufficiently close to $\calP^*$ as required in \Cref{Lemm:local-finite-length} (see the inequality \eqref{Eq:initial-cond}). Therefore, the whole sequence $\{\calP^k\}_{k\in\NN}$ converges according to \Cref{Lemm:local-finite-length}. Since $\calP^*$ is a limit point of $\{\calP^k\}_{k\in\NN}$, we have $\calP^k \rightarrow \calP^*$.
	
	Next, we show $\calP^*$ is a critical point of $\bcL$.
	By \Cref{Coro:square-summable}(c), $\lim_{k\rightarrow \infty}\|\calP^k - \calP^{k-1}\|_F = 0$. Furthermore, by \Cref{Lemm:BCD-grad-bound1},
	\[
	\lim_{k\rightarrow \infty} \mathrm{dist}(\zero, \partial \bcL(\calP^k)) = 0,
	\]
	which implies that any limit point is a critical point. Therefore, we prove the global convergence of the sequence generated by \Cref{alg:BCD-3-split}.
	
	The convergence to a global minimum is a straightforward variant of \Cref{Lemm:local-finite-length}.

    The ${\cal O}(1/k)$ rate of convergence is a direct claim according to the proof of \Cref{Lemm:BCD-grad-bound1} and \Cref{Coro:square-summable}(c).
	
	The proof of the convergence of \Cref{alg:BCD-2-split} is similar to that of \Cref{alg:BCD-3-split}. We give a brief description about this.
	Note that in \Cref{alg:BCD-2-split}, all blocks of variables are updated via the proximal strategies (or, prox-linear strategy for $\bV_N$-block). Thus, it is easy to show the similar descent inequality, i.e.,
	\begin{align}
	\label{Eq:BCD-2-split-descent}
	\calL(\calQ^{k-1}) - \calL(\calQ^{k}) \geq a \|\calQ^k - \calQ^{k-1}\|_F^2,
	\end{align}
	for some $a>0$.
	Then similar to the proof of \Cref{Lemm:BCD-grad-bound1}, we can establish the following inequality via checking the optimality conditions of all subproblems in \Cref{alg:BCD-2-split}, i.e.,
	\begin{align}
	\label{Eq:BCD-2-split-grad-bound}
	\mathrm{dist}(\zero,\partial \calL (\calQ^k)) \leq b \|\calQ^k - \calQ^{k-1}\|_F,
	\end{align}
	for some $b>0$.
	By \eqref{Eq:BCD-2-split-descent}, \eqref{Eq:BCD-2-split-grad-bound} and the \KL property of $\calL$ (by \Cref{Thm:KL-property}), the global convergence of \Cref{alg:BCD-2-split} can be proved via a similar proof procedure of \Cref{alg:BCD-3-split}.
	\hfill$\Box$
\end{proof}

\subsubsection{Condition (b) or (c) implies condition (a)}
\label{app:estab-boundedness}
\begin{lemma}
	\label{Lemm:boundednes}
	Under condition (b) or condition (c) of \Cref{Thm:BCD-ConvThm},
	$\calP^k$ is bounded for any $k\in \mathbb{N}$, and thus, there exists a convergent subsequence.
\end{lemma}
\begin{proof}
	We first show the boundedness of the sequence as well as the subsequence convergence  under \textbf{condition (b)} of \Cref{Thm:BCD-ConvThm}, then under \textbf{condition (c)} of \Cref{Thm:BCD-ConvThm}.
	
	\begin{enumerate}
		\item \textbf{Condition (b) implies condition (a):}
		We first establish the boundedness of $\bW_i^k$, $i=1,\ldots,N$. Then, recursively, we establish the boundedness of $\bU_i^k$ via the boudedness of $\bW_i^k$ and $\bV_{i-1}^k$ (noting that $\bV_0^k \equiv \bX$), followed by that of $\bV_i^k$ via the boundedness of $\bU_i^k$, $i=1,\ldots,N$.
		
		\begin{enumerate}[label=(\arabic*)]
			\item Boundedness of $\bW_i^k\ (i=1,\ldots,N)$: By \Cref{Lemm:BCD-suff-desc1}, $\bcL(\calP^k) < \infty$ for all $k\in \mathbb{N}$. Noting that each term of $\bcL$ is nonnegative, thus, $0\leq r_i(\bW_i^k)<\infty$ for any $k\in \mathbb{N}$ and $i=1,\ldots, N$. By the coercivity of $r_i$, $\bW_i^k$ is boundedness for any $k\in \mathbb{N}$ and $i=1,\ldots, N$.
			
			In the following, we establish the boundedness of $\bU_i^k$ for any $k\in \mathbb{N}$ and $i=1,\ldots, N$.
			
			\item $i=1$: Since $\bcL(\calP^k) < \infty$, then $\|\bU_1^k - \bW_1^k \bX\|_F^2 <\infty$ for any $k\in \mathbb{N}$. By the boundedness of $\bW_1^k$ and the coercivity of the function $\|\cdot\|_F^2$, we have the boundedness of $\bU_1^k$ for any $k\in \mathbb{N}$.
			Then we show the boundedness of $\bV_1^k$ by the boundedness of $\bU_1^k$. Since $\bcL(\calP^k) < \infty$, then $\|\bV_1^k - \sigma_1(\bU_1^k)\|_F^2 <\infty$ for any $k\in \mathbb{N}$. By the Lipschitz continuity of $\sigma_1$ and the boundedness of $\bU_1^k$, $\sigma_1(\bU_1^k)$ is uniformly bounded for any $k\in \mathbb{N}$. Thus, by the coercivity of $\|\cdot\|_F^2$, $\bV_1^k$ is bounded for any $k\in \mathbb{N}$.
			
			\item $i>1$: Recursively, we show that the boundedness of $\bW_i^k$ and $\bV_{i-1}^k$ implies the boundedness of $\bU_i^k$, and then the boundedness of $\bV_i^k$ from $i=2$ to $N$.
			
			Now, we assume that the boundedness of $\bV_{i-1}^k$ has been established. Similar to \textbf{(2)}, the boundedness of $\bU_i^k$ is guaranteed by $\|\bU_i^k - \bW_i^k \bV_{i-1}^k\|_F^2 <\infty$ and the boundedness of $\bW_i^k$ and $\bV_{i-1}^k$. The boundedness of $\bV_i^k$ is guaranteed by $\|\bV_i^k - \sigma_i(\bU_i^k)\|_F^2<\infty$ and the boundedness of $\bU_i^k$, as well as the Lipschitz continuity of $\sigma_i$.
	
			As a consequence, we prove the boundedness of $\{\calP^k\}_{k\in\NN}$ under condition (b), which implies the subsequence convergent.
		\end{enumerate}

		\item \textbf{Condition (c) implies condition (a):} By \Cref{Lemm:BCD-suff-desc1} and the finite initialization assumption, we have
		\[{\bcL(\calP^k)} \leq {\bcL}(\calP^0)<\infty,\]
		which implies the boundedness of $\calP^k$ due to the coercivity of $\bcL$ (i.e., condition (c)), and thus, there exists a convergent subsequence.
	\end{enumerate}
	
	This completes the proof of this lemma.
\hfill$\Box$
\end{proof}

\section{Proof of \Cref{Coro:BCD-resnet}}
\label{sc:proof-convergence-BCD-ResNets}

\begin{algorithm}[t]
{\small
\begin{algorithmic}\caption{BCD for DNN Training with ResNets \eqref{Eq:dnn-resnet}}\label{alg:BCD-resnet}
\STATE {\bf Samples}: $\bX \in \RR^{d_0 \times n}$, $\bY \in \RR^{d_N \times n}$, $\bV_0^k \equiv \bV_0 := \bX$
\STATE {\bf Initialization}: $\{\bW_i^0, \bV_i^0, \bU_i^0\}_{i=1}^N$
\STATE {\bf Parameters:} $\gamma>0$, $\alpha>0$
\smallskip
\FOR{$k=1,\ldots$}
\STATE $\bV_N^k = \argmin_{\bV_N} \ \{ s_N(\bV_N) + \calR_n(\bV_N; \bY) + \frac{\gamma}{2} \|\bV_N - \bV_{N-1}^{k-1} - \bU_N^{k-1}\|_F^2 + \frac{\alpha}{2} \|\bV_N - \bV_N^{k-1}\|_F^2\}$
\STATE $\bU_N^k = \argmin_{\bU_N} \ \{\frac{\gamma}{2}(\|\bV_N^k - \bV_{N-1}^{k-1} -  \bU_N\|_F^2 + \|\bU_N - \bW_N^{k-1}\bV_{N-1}^{k-1}\|_F^2)\}$
\STATE $\bW_N^k = \argmin_{\bW_N} \ \{r_N(\bW_N)+\frac{\gamma}{2} \|\bU_N^k - \bW_N \bV_{N-1}^{k-1}\|_F^2 + \frac{\alpha}{2}\|\bW_N - \bW_N^{k-1}\|_F^2\}$
\FOR{$i= N-1,\ldots, 1$}
\STATE $\bV_i^k = \argmin_{\bV_i} \ \{s_i(\bV_i) + \frac{\gamma}{2} (\|\bV_i - \bV_{i-1}^{k-1} - \sigma_i(\bU_i^{k-1})\|_F^2 + \|\bV_{i+1}^k - \bV_{i} - \sigma_{i+1}(\bU_{i+1}^{k})\|_F^2+ \|\bU_{i+1}^k - \bW_{i+1}^k \bV_i\|_F^2)\}$
\STATE $\bU_i^k = \argmin_{\bU_i}\ \{\frac{\gamma}{2}(\|\bV_i^k - \bV_{i-1}^{k-1} - \sigma_i(\bU_i)\|_F^2 + \|\bU_i - \bW_i^{k-1}\bV_{i-1}^{k-1}\|_F^2) + \frac{\alpha}{2}\|\bU_i -\bU_i^{k-1}\|_F^2\}$
\STATE $\bW_i^k = \argmin_{\bW_i} \{ r_i(\bW_i) + \frac{\gamma}{2} \|\bU_i^k - \bW_iV_{i-1}^{k-1}\|_F^2 + \frac{\alpha}{2} \|\bW_i - \bW_i^{k-1}\|_F^2\}$
\ENDFOR
\ENDFOR
\end{algorithmic}}
\end{algorithm}

The proof of \Cref{Coro:BCD-resnet} is very similar to those of \Cref{Lemm:BCD-suff-desc1} and \Cref{Thm:BCD-ConvThm}, by noting that the updates are slightly different.
In the following, we present the proof of \Cref{Coro:BCD-resnet}.

\Cref{Lemm:BCD-suff-desc1} still holds for \Cref{alg:BCD-resnet} via replacing $\bcL$ with $\bcL_{\res}$,
which is stated as the following lemma.

\begin{lemma}
\label{Lemm:BCD-suff-desc-resnet}
Let $\{\{\bW_i^k, \bV_i^k, \bU_i^k\}_{i=1}^N\}_{k\in\NN}$ be a sequence generated by the BCD method (\Cref{alg:BCD-resnet}) for the DNN training with ResNets,
then for any $\gamma>0, \alpha>0$,
\begin{multline}
\label{Eq:BCD-suff-desc-resnet}
\bcL_{\res}\left( \{\bW_i^k,\bV_i^k,\bU_i^k\}_{i=1}^N\right)
\leq \bcL_{\res}\left( \{\bW_i^{k-1},\bV_i^{k-1},\bU_i^{k-1}\}_{i=1}^N\right)  \\
- a\sum_{i=1}^N \left[  \|\bW_i^k - \bW_i^{k-1}\|_F^2 + \|\bV_i^k - \bV_i^{k-1}\|_F^2 + \|\bU_i^k - \bU_i^{k-1}\|_F^2 \right] ,
\end{multline}
and $\left\lbrace \bcL_{\res}\left( \{\bW_i^k,\bV_i^k,\bU_i^k\}_{i=1}^N\right) \right\rbrace _{k\in\NN}$ converges to some $\bcL_{\res}^*,$
where $a:=\min\left\lbrace \frac{\alpha}{2},\frac{\gamma}{2}\right\rbrace $.

\end{lemma}

\begin{proof}
The proof of this lemma is the same as that of \Cref{Lemm:BCD-suff-desc1}.
\hfill$\Box$
\end{proof}

In the ResNets case, \Cref{Lemm:BCD-grad-bound1} should be revised as the following lemma.

\begin{lemma}
\label{Lemm:BCD-grad-bound-resnet}
Let $\{\{\bW_i^k, \bV_i^k, \bU_i^k\}_{i=1}^N\}_{k\in\NN}$ be a sequence generated by \Cref{alg:BCD-resnet}.
Under Assumptions of \Cref{Coro:BCD-resnet},
let $\bar{b}:= \max\{\alpha + \gamma L, \alpha + \gamma {\cal B}, 2\gamma (1+{\cal B} + {\cal B}^2), \gamma (1+L{\cal B} + 2{\cal B} + 2{\cal B}^2)\}$,
then
\begin{equation}
\label{Eq:BCD-grad-bound-resnet}
\mathrm{dist}(\zero,\partial \bcL_{\res}(\{\bW_i^k,\bV_i^k,\bU_i^k\}_{i=1}^N))
\leq \bar{b} \sum_{i=1}^N \left[\|\bW_i^k - \bW_i^{k-1}\|_F+ \|\bV_i^k - \bV_i^{k-1}\|_F + \|\bU_i^k - \bU_i^{k-1}\|_F\right],
\end{equation}
where
\begin{align*}
&\partial \bcL_{\res}(\{\bW_i^k,\bV_i^k,\bU_i^k\}_{i=1}^N)
:= (\{\partial_{\bW_i}\bcL_{\res}, \partial_{\bV_i} \bcL_{\res}, \partial_{\bU_i} \bcL_{\res} \}_{i=1}^N)(\{\bW_i^k,\bV_i^k,\bU_i^k\}_{i=1}^N).
\end{align*}
\end{lemma}

\begin{proof}
From updates of \Cref{alg:BCD-resnet},
\begin{align*}
& \zero \in \partial s_N(\bV_N^k) + \partial {\cal R}_n(\bV_N^k; \bY) + \gamma (\bV_N^k - \bV_{N-1}^{k-1} - \bU_N^{k-1})
 + \alpha (\bV_N^k - \bV_N^{k-1}),\\
& \zero = \gamma (\bU_N^k + \bV_{N-1}^{k-1}- \bV_N^k)+ \gamma (\bU_N^k - \bW_N^{k-1}\bV_{N-1}^{k-1}), \\
& \zero \in \partial r_N(\bW_N^k) + \gamma (\bW_N^k \bV_{N-1}^{k-1} - \bU_N^k) {\bV_{N-1}^{k-1}}^\top
+ \alpha (\bW_N^k - \bW_N^{k-1}),\\
&\text{for}\ i= N-1,\ldots, 1, \nonumber\\
& \zero \in \partial s_i(\bV_i^k) + \gamma (\bV_i^k - \bV_{i-1}^{k-1} - \sigma_i(\bU_i^{k-1}))
- \gamma (\bV_{i+1}^k - \bV_i^k - \sigma_{i+1}(\bU_{i+1}^k))
+ \gamma {\bW_{i+1}^k}^\top (\bW_{i+1}^k \bV_i^k - \bU_{i+1}^k), \\
& \zero \in  \gamma [(\sigma_i(\bU_i^k) + \bV_{i-1}^{k-1} - \bV_i^k)\odot \partial \sigma_i(\bU_i^k)]
 + \gamma (\bU_i^k - \bW_i^{k-1} \bV_{i-1}^{k-1}) + \alpha(\bU_i^k - \bU_i^{k-1}), \\
& \zero \in  \partial r_i(\bW_i^k) + \gamma (\bW_i^k \bV_{i-1}^{k-1}- \bU_i^k){\bV_{i-1}^{k-1}}^\top
+ \alpha (\bW_i^k - \bW_i^{k-1}),
\end{align*}
where $\bV_0^k \equiv \bV_0 = \bX$, for all $k$, and $\odot$ is the Hadamard product.
By the above relations, we have
\begin{align*}
&-\alpha (\bV_N^k - \bV_N^{k-1}) - \gamma (\bV_{N-1}^k - \bV_{N-1}^{k-1})-\gamma (\bU_N^k - \bU_N^{k-1}) \\
&\in \partial s_N(\bV_N^k) + \partial {\cal R}_n(\bV_N^k; \bY) + \gamma (\bV_N^k - \bV_{N-1}^k- \bU_N^{k})
= \partial_{\bV_N} \bcL_{\res}(\{\bW_i^k,\bV_i^k,\bU_i^k\}_{i=1}^N),\\
&-\gamma (\bW_N^k - \bW_N^{k-1})\bV_{N-1}^k - \gamma \bW_N^{k-1}(\bV_{N-1}^k - \bV_{N-1}^{k-1})
+ \gamma (\bV_{N-1}^k - \bV_{N-1}^{k-1}) \\
&= \gamma (\bU_N^k + \bV_{N-1}^k - \bV_N^k)+ \gamma (\bU_N^k - \bW_N^{k} \bV_{N-1}^{k})
= \partial_{\bU_N} \bcL_{\res}(\{\bW_i^k, \bV_i^k, \bU_i^k\}_{i=1}^N),\\
&\gamma {\bW_N^k}\left[ \bV_{N-1}^k ({\bV_{N-1}^k} - {\bV_{N-1}^{k-1}})^\top + (\bV_{N-1}^k - \bV_{N-1}^{k-1}){\bV_{N-1}^{k-1}}^\top\right]
- \gamma \bU_N^k({\bV_{N-1}^k} - {\bV_{N-1}^{k-1}})^\top - \alpha (\bW_N^k - \bW_N^{k-1}) \\
& \qquad\qquad\qquad\qquad\qquad\qquad\quad\in \partial r_N(\bW_N^k) + \gamma (\bW_N^k \bV_{N-1}^k - \bU_N^k){\bV_{N-1}^k}^\top
= \partial_{\bW_N}\bcL_{\res}(\{\bW_i^k,\bV_i^k,\bU_i^k\}_{i=1}^N),\\
\end{align*}
For $i= N-1,\ldots, 1$, and $\bm{\xi}_i^k \in \partial \sigma_i(\bU_i^k)$,
\begin{align*}
&-\gamma (\bV_{i-1}^k - \bV_{i-1}^{k-1})-\gamma (\sigma_i(\bU_i^k) - \sigma_i(\bU_i^{k-1}))\\
&\qquad\quad \in \partial s_i(\bV_i^k) + \gamma (\bV_i^k - \bV_{i-1}^{k} - \sigma_i(\bU_i^{k}))
- \gamma (\bV_{i+1}^k - \bV_i^k - \sigma_{i+1}(\bU_{i+1}^k))
+ \gamma {\bW_{i+1}^k}^\top (\bW_{i+1}^k \bV_i^k - \bU_{i+1}^k)\\
&\qquad\quad =  \partial_{\bV_i} \bcL_{\res}(\{\bW_i^k,\bV_i^k,\bU_i^k\}_{i=1}^N),\\
&-\gamma \bW_i^{k-1}(\bV_{i-1}^k - \bV_{i-1}^{k-1}) -\gamma (\bW_i^k - \bW_i^{k-1})\bV_{i-1}^k
-\alpha(\bU_i^k -\bU_i^{k-1}) + \gamma (\bV_{i-1}^k - \bV_{i-1}^{k-1})\odot \bm{\xi}_i^k \nonumber\\
&\qquad\qquad\qquad\quad\in \gamma [(\sigma_i(\bU_i^k) + \bV_{i-1}^k -\bV_i^k)\odot \partial \sigma_i(\bU_i^k)] + \gamma (\bU_i^k - \bW_i^{k} \bV_{i-1}^{k}) = \partial_{\bU_i} \bcL_{\res}(\{\bW_i^k, \bV_i^k, \bU_i^k\}_{i=1}^N),
\end{align*}
\begin{align*}
& \gamma {\bW_i^k}\left[ \bV_{i-1}^k ({\bV_{i-1}^k} - {\bV_{i-1}^{k-1}})^\top + (\bV_{i-1}^k - \bV_{i-1}^{k-1}){\bV_{i-1}^{k-1}}^\top\right]
- \gamma \bU_i^k({\bV_{i-1}^k} - {\bV_{i-1}^{k-1}})^\top  - \alpha (\bW_i^k - \bW_i^{k-1}) \nonumber\\
& \qquad\qquad\qquad\qquad\qquad\qquad\qquad\qquad\qquad\qquad\qquad\qquad\qquad\qquad\qquad\in \partial r_i(\bW_i^k) + \gamma (\bW_i^k \bV_{i-1}^k - \bU_i^k){\bV_{i-1}^k}^\top \nonumber\\
&\qquad\qquad\qquad\qquad\qquad\qquad\qquad\qquad\qquad\qquad\qquad\qquad\qquad\qquad\qquad = \partial_{\bW_i}\bcL_{\res}(\{\bW_i^k,\bV_i^k,\bU_i^k\}_{i=1}^N).
\end{align*}
From the above relations, the uniform boundedness of the generated sequence (whose bound is ${\cal B}$) and the Lipschitz continuity of the activation function by the hypothesis of this lemma, we have $\|\bm{\xi}_i^k\|\leq L{\cal B}$, and further we get \eqref{Eq:BCD-grad-bound-resnet}.
\hfill$\Box$
\end{proof}

\begin{proof}[Proof of \Cref{Coro:BCD-resnet}]
The proof of this theorem is very similar to that of \Cref{Thm:BCD-ConvThm}.
First, similar to \Cref{Thm:KL-property}, it is easy to show that $\bcL_{\res}$ is also a \KL function.
Then, based on \Cref{Lemm:BCD-suff-desc-resnet} and \Cref{Lemm:BCD-grad-bound-resnet}, and the \KL property of $\bcL_{\res}$, we can prove this corollary by \citet[Theorem 2.9]{Attouch2013}.
The other claims of this theorem follow from the same proof of \Cref{Thm:BCD-ConvThm}.
When the prox-linear strategy is adopted for the $\bV_N$-update, the claims of \Cref{Coro:BCD-resnet} can be proved via following the same proof of \Cref{Coro:globalconv-prox-linear1}.
\hfill$\Box$
\end{proof}

\section{Closed form solutions of some subproblems}
\label{sc:closed-form-solution}
In this section, we provide the closed form solutions to the ReLU involved subproblem and the hinge loss involved subproblem.

\subsection{Closed form solution to ReLU-subproblem}
\label{app:proof_lemma_1}
From \Cref{alg:BCD-3-split}, when $\sigma_i$ is ReLU, then the $\bU_i^k$-update actually reduces to the following one-dimensional minimization problem,
\begin{align}
\label{Eq:ReLU-min}
u^* = \argmin_u f(u) := \frac{1}{2}(\sigma(u)-a)^2 + \frac{\gamma}{2}(u-b)^2,
\end{align}
where $\sigma(u) = \max\{0,u\}$ and $\gamma>0$.
The solution to the above one-dimensional minimization problem can be presented in the following lemma.
\begin{lemma}
	\label{Lemm:solution-relu-min}
	The optimal solution to Problem \eqref{Eq:ReLU-min} is shown as follows
	\begin{equation*}
	\mathrm{prox}_{\frac{1}{2\gamma}(\sigma(\cdot)-a)^2}(b) =
	\begin{cases}
	\dfrac{a+\gamma b}{1+\gamma}, \ & \text{if}\ a + \gamma b \geq 0, \ b \geq 0, \\[4mm]
	\dfrac{a+\gamma b}{1+\gamma}, \ & \text{if}\ -(\sqrt{\gamma(\gamma+1)}-\gamma)a \leq \gamma b <0,\\[4mm]
	b, \ & \text{if}\ -a \leq \gamma b \leq - (\sqrt{\gamma (\gamma +1)} - \gamma)a<0, \\
	\min\{b,0\}, \ & \text{if} \ a + \gamma b <0.
	\end{cases}
	\end{equation*}
\end{lemma}
\begin{proof}
	
	In the following, we divide this into two cases.
	\begin{enumerate}[label=(\alph*)]
		\item $u\ge0$: In this case,
		\[
		f(u) = \frac{1}{2} (u-a)^2 + \frac{\gamma}{2}(u-b)^2.
		\]
		It is easy to check that
		\begin{equation}
		\label{Eq:solution1-relu}
		u^* =
		\begin{cases}
		\dfrac{a+\gamma b}{1+\gamma}, \ & \text{if}\ a+\gamma b \geq 0\\[2mm]
		0, \ & \text{if}\ a+\gamma b < 0
		\end{cases},
		\end{equation}
		and
		\[
		f\left(\frac{a+\gamma b}{1+\gamma}\right) = \frac{\gamma}{2(1+\gamma)}(b-a)^2, \quad f(0) = \frac{1}{2}a^2 + \frac{\gamma}{2} b^2.
		\]
		
		\item $u<0$: In this case,
		\[
		f(u) = \frac{1}{2} a^2 + \frac{\gamma}{2}(u-b)^2.
		\]
		It is easy to check that
		\begin{equation}
		\label{Eq:solution2-relu}
		u^* =
		\begin{cases}
		0, \ & \text{if}\ b \geq 0 \\
		b, \ & \text{if}\ b < 0
		\end{cases}
		,
		\end{equation}
		and
		\[
		f(b) = \frac{1}{2}a^2, \quad f(0) = \frac{1}{2}a^2 + \frac{\gamma}{2} b^2.
		\]
		
	\end{enumerate}

	Based on \eqref{Eq:solution1-relu} and \eqref{Eq:solution2-relu}, we obtain the solution to Problem \eqref{Eq:ReLU-min} by considering the following four cases.
	
	\begin{enumerate}
		\item $a+\gamma b \geq 0, b \geq 0$:
	In this case, we need to compare the values $f\left( \frac{a+\gamma b}{1+\gamma}\right)  = \frac{\gamma}{2(1+\gamma)}(b-a)^2$ and $f(0) = \frac{1}{2}a^2 + \frac{\gamma}{2} b^2$. It is obvious that
	\[
	u^* = \frac{a+\gamma b}{1+\gamma}.
	\]
	
	\item $a+\gamma b \geq 0, b < 0$:
	In this case, we need to compare the values $f\left( \frac{a+\gamma b}{1+\gamma}\right)  = \frac{\gamma}{2(1+\gamma)}(b-a)^2$ and $f(b) = \frac{1}{2}a^2$.
	By the hypothesis of this case, it is obvious that $a>0$.
	We can easily check that
	\begin{equation*}
	u^* =
	\begin{cases}
	\dfrac{a+\gamma b}{1+\gamma}, & \text{if}\ -(\sqrt{\gamma(\gamma+1)}-\gamma)a \leq \gamma b <0, \\[2mm]
	b, & \text{if}\ -a \leq \gamma b \leq - (\sqrt{\gamma (\gamma +1)} - \gamma)a<0.
	\end{cases}
	\end{equation*}
	
	\item $a+\gamma b < 0, b \geq 0$: It is obvious that
	\[u^* = 0.\]
	
	\item $a+\gamma b < 0, b < 0$: It is obvious that
	\[u^* = b.\]
	\end{enumerate}
	
	Thus, the solution to Problem \eqref{Eq:ReLU-min} is
	\begin{equation*}
	\mathrm{prox}_{\frac{1}{2\gamma}(\sigma(\cdot)-a)^2}(b)  =
	\begin{cases}
	\dfrac{a+\gamma b}{1+\gamma},  & \text{if}\ a + \gamma b \geq 0, \ b \geq 0, \\[4mm]
	\dfrac{a+\gamma b}{1+\gamma},  & \text{if}\ -(\sqrt{\gamma(\gamma+1)}-\gamma)a \leq \gamma b <0,\\[4mm]
	b,  & \text{if}\ -a \leq \gamma b \leq - (\sqrt{\gamma (\gamma +1)} - \gamma)a<0, \\
	\min\{b,0\},  & \text{if} \ a + \gamma b <0.
	\end{cases}
	\end{equation*}
	\hfill$\Box$
\end{proof}

\subsection{The closed form of the proximal operator of hinge loss}
\label{app:prox}
Consider the following optimization problem
\begin{align}
\label{Eq:hinge-min}
u^* = \argmin_u g(u) := \max\{0,1-a\cdot u\} + \frac{\gamma}{2}(u-b)^2,
\end{align}
where $\gamma>0$.
\begin{lemma}
	\label{Lemm:solution-hinge-min}
	The optimal solution to Problem \eqref{Eq:hinge-min} is shown as follows
	\begin{equation*}
	\mathrm{hinge}_{\gamma}(a,b) =
	\begin{cases}
	b, &\ \text{if} \ a=0,\\
	b+\gamma^{-1}a, &\ \text{if} \ a \neq 0 \ \text{and}\ ab\leq 1-\gamma^{-1}a^2,\\
	a^{-1}, &\ \text{if} \ a \neq 0 \ \text{and}\ 1-\gamma^{-1}a^2 <ab<1,\\
	b, &\ \text{if} \ a \neq 0 \ \text{and}\ ab\geq 1.
	\end{cases}
	\end{equation*}
\end{lemma}

\begin{proof}
	We consider the problem in the following three different cases: (1) $a>0$, (2) $a=0$ and (3) $a<0$.
	\begin{enumerate}[label=(\arabic*)]
		\item $a>0$: In this case,
		\begin{equation*}
		g(u)
		=
		\begin{cases}
		1-au+\dfrac{\gamma}{2}(u-b)^2, & \text{if} \ u<a^{-1},\\[2mm]
		\dfrac{\gamma}{2}(u-b)^2,  & \text{if} \ u\geq a^{-1}.
		\end{cases}
		\end{equation*}
		It is easy to show that the solution to the problem is
		\begin{align}
		\label{Eq:hing-sol-part1}
		u^* =
		\begin{cases}
		b+\gamma^{-1}a, &\ \text{if} \ a >0 \ \text{and}\ b\leq a^{-1}-\gamma^{-1}a,\\
		a^{-1}, &\ \text{if} \ a >0 \ \text{and}\ a^{-1}-\gamma^{-1}a <b<a^{-1},\\
		b, &\ \text{if} \ a >0 \ \text{and}\ b\geq a^{-1}.
		\end{cases}
		\end{align}
		
		\item $a=0$: It is obvious that
		\begin{align}
		\label{Eq:hing-sol-part2}
		u^* =b.
		\end{align}
		
		\item $a<0$: Similar to (1),
		\begin{align*}
		g(u)
		=
		\begin{cases}
		1-au+\dfrac{\gamma}{2}(u-b)^2, \ & u\geq a^{-1},\\[2mm]
		\dfrac{\gamma}{2}(u-b)^2, \ & {u < a^{-1}}.
		\end{cases}
		\end{align*}
		Similarly, it is easy to show that the solution to the problem is
		\begin{align}
		\label{Eq:hing-sol-part3}
		u^* =
		\begin{cases}
		b+\gamma^{-1}a, &\ \text{if} \ a <0 \ \text{and}\ b\geq a^{-1}-\gamma^{-1}a,\\
		a^{-1}, &\ \text{if} \ a <0 \ \text{and}\ a^{-1}<b<a^{-1}-\gamma^{-1}a ,\\
		b, &\ \text{if} \ a <0 \ \text{and}\ b \leq a^{-1}.
		\end{cases}
		\end{align}		
	\end{enumerate}
Thus, we finish the proof of this lemma.
	\hfill$\Box$	
\end{proof}

\section{BCD vs. SGD for training ten-hidden-layer MLPs}
\label{sc:BCD-SGD-10layerMLP}

In this experiment, we attempt to verify the capability of BCD for training MLPs with many layers. Reproducible PyTorch codes can be found at: \url{https://github.com/timlautk/BCD-for-DNNs-PyTorch} or \url{https://github.com/yao-lab/BCD-for-DNNs-PyTorch}.

Specifically, we consider the DNN training model \eqref{Eq:dnn-admm} with ReLU activation, the squared loss, and the network architecture being an MLPs with ten hidden layers, on the MNIST data set.
The specific settings were summarized as follows:
\begin{enumerate}[label = (\alph*)]
	\item For the MNIST data set, we implemented a 784-(600$\times$10)-10 MLPs (i.e., the input dimension $d_0 = 28\times 28 = 784$,  the output dimension $d_{11} = 10$, and the numbers of hidden units are all 600), and set $\gamma=\alpha=1$ for BCD.
	The sizes of training and test samples are 60000 and 10000, respectively.

	\item The learning rate of SGD is 0.001 (a very conservative learning rate to see if SGD can train the DNNs). More greedy learning rates such as 0.01 and 0.05 have also been used, and similar failure of training is also observed.
	
	\item For each experiment,
	we used the same mini-batch sizes (512) and initializations for all algorithms.
	Specifically, all the weights $\{\bW_i\}_{i=1}^N$ are initialized from a Gaussian distribution with a standard deviation of $0.01$ and the bias vectors are initialized as vectors of all $0.1$, while the auxiliary variables $\{\bU_i\}_{i=1}^N$ and state variables $\{\bV_i\}_{i=1}^N$ are initialized by a single forward pass.
	
\end{enumerate}

Under these settings,
we plot the curves of training accuracy (acc.) and test accuracy (acc.) of BCD and SGD as shown in \Cref{Fig:acc_deep}.
According to \Cref{Fig:acc_deep}, vanilla SGD usually fails to train such deeper MLPs since it suffers from the vanishing gradient issue \citep{Goodfellow-et-al-2016},
whereas BCD still works and achieves a moderate accuracy within a few epochs.

\end{document}